\documentclass[twoside,english,american,DIV12,listtotoc,bibtotoc,idxtotoc]{scrartcl}
\usepackage[T1]{fontenc}
\usepackage[latin1]{inputenc}
\usepackage{babel}
\usepackage{prettyref}
\usepackage{amsthm}
\usepackage{amsmath}
\usepackage{amssymb}
\usepackage[unicode=true,pdfusetitle,
 bookmarks=true,bookmarksnumbered=false,bookmarksopen=false,
 breaklinks=false,pdfborder={0 0 1},backref=false,colorlinks=false]
 {hyperref}

\makeatletter
\theoremstyle{plain}
\newtheorem{thm}{\protect\theoremname}[section]
  \theoremstyle{plain}
  \newtheorem{prop}[thm]{\protect\propositionname}
  \theoremstyle{definition}
  \newtheorem{defn}[thm]{\protect\definitionname}
  \theoremstyle{plain}
  \newtheorem{lem}[thm]{\protect\lemmaname}
  \theoremstyle{plain}
  \newtheorem{cor}[thm]{\protect\corollaryname}
  \theoremstyle{remark}
  \newtheorem{rem}[thm]{\protect\remarkname}

\usepackage{helvet}
\usepackage[T1]{fontenc}

\usepackage{amsmath}
\usepackage{setspace}
\usepackage{amssymb}
\usepackage{enumerate}
\usepackage{url}
\usepackage{prettyref}

\makeatletter

\newrefformat{prop}{Proposition \ref{#1}}
\newrefformat{lem}{Lemma \ref{#1}}
\newrefformat{thm}{Theorem \ref{#1}}
\newrefformat{cor}{Corollary \ref{#1}}
\newrefformat{rem}{Remark \ref{#1}}
\newrefformat{eq}{(\ref{#1})}
\newrefformat{item}{(\ref{#1})}

\usepackage[T1]{fontenc}    

\usepackage{a4wide}        
\usepackage{tu-preprint}
\usepackage{amsmath}

\usepackage{euscript}
\usepackage{mathtools}
\allowdisplaybreaks[1] 
\usepackage{amsfonts}
\newenvironment{keywords}{ \noindent\footnotesize\textbf{Keywords and phrases:}}{}

\newenvironment{class}{\noindent\footnotesize\textbf{Mathematics subject classification 2010:}}{}

\usepackage{color,tu-preprint}

\newcommand*{\trace}{\operatorname{trace}}

\newcommand*{\dive}{\operatorname{div}}

\newcommand{\R}{\mathbb{R}}
\newcommand*{\Grad}{\operatorname{Grad}}
\newcommand*{\Dive}{\operatorname{Div}}

\newcommand*{\grad}{\operatorname{grad}}

\DeclareMathAccent{\Circ}{\mathalpha}{operators}{"17}

\newcommand{\spt}{\operatorname{spt}}

\renewcommand{\Re}{\operatorname{\mathfrak{Re}}}

\renewcommand{\tilde}{\widetilde}
\renewcommand*{\epsilon}{\varepsilon}
\renewcommand*{\rho}{\varrho}

\makeatother

\usepackage{babel}
\institut{Institut f\"ur Analysis}

\preprintnumber{MATH-AN-07-2013}

\preprinttitle{Well-posedness of Non-autonomous Evolutionary Inclusions.}

\author{Sascha Trostorff \& Maria Wehowski}

\theoremstyle{definition}
\newtheorem*{cond}{Conditions}

\AtBeginDocument{
  
}

\makeatother

  \addto\captionsamerican{\renewcommand{\corollaryname}{Corollary}}
  \addto\captionsamerican{\renewcommand{\definitionname}{Definition}}
  \addto\captionsamerican{\renewcommand{\lemmaname}{Lemma}}
  \addto\captionsamerican{\renewcommand{\propositionname}{Proposition}}
  \addto\captionsamerican{\renewcommand{\remarkname}{Remark}}
  \addto\captionsamerican{\renewcommand{\theoremname}{Theorem}}
  \addto\captionsenglish{\renewcommand{\corollaryname}{Corollary}}
  \addto\captionsenglish{\renewcommand{\definitionname}{Definition}}
  \addto\captionsenglish{\renewcommand{\lemmaname}{Lemma}}
  \addto\captionsenglish{\renewcommand{\propositionname}{Proposition}}
  \addto\captionsenglish{\renewcommand{\remarkname}{Remark}}
  \addto\captionsenglish{\renewcommand{\theoremname}{Theorem}}
  \providecommand{\corollaryname}{Corollary}
  \providecommand{\definitionname}{Definition}
  \providecommand{\lemmaname}{Lemma}
  \providecommand{\propositionname}{Proposition}
  \providecommand{\remarkname}{Remark}
\providecommand{\theoremname}{Theorem}

\begin{document}
\makepreprinttitlepage

\author{ Sascha Trostorff \& Maria Wehowski \\ Institut f\"ur Analysis, Fachrichtung Mathematik\\ Technische Universit\"at Dresden\\ Germany\\ sascha.trostorff@tu-dresden.de\\ maria.wehowski@tu-dresden.de}

\title{Well-posedness of Non-autonomous Evolutionary Inclusions.}

\maketitle
\begin{abstract} \textbf{Abstract.} A class of non-autonomous differential
inclusions in a Hilbert space setting is considered. The well-posedness
for this class is shown by establishing the mappings involved as maximal
monotone relations. Moreover, the causality of the so established
solution operator is addressed. The results are exemplified by the
equations of thermoplasticity with time dependent coefficients and
by a non-autonomous version of the equations of viscoplasticity with
internal variables.\end{abstract}


\begin{class} 34G25, 35F60, 37B55, 46N20 \end{class}

\newpage

\tableofcontents{} 

\newpage

\section{Introduction}

As it was pointed out in \cite{Picard}, the classical equations of
mathematical physics share a common form, namely 
\[
\partial_{0}v+Au=f,
\]
where $\partial_{0}$ denotes differentiation with respect to time
and $A:D(A)\subseteq H\to H$ is a suitable linear operator on a Hilbert
space $H$. The equation needs to be completed by a constitutive relation
linking the unknowns $u$ and $v$. We consider a certain class of
such constitutive relations, which actually occurs frequently in mathematical
physics. It can be written in the form 
\[
\partial_{0}v=\partial_{0}M_{0}u+M_{1}u,
\]
where $M_{0},M_{1}\in L(H)$ with $M_{0}$ selfadjoint and strictly
positive definite on its range and $\Re M_{1}=\frac{1}{2}(M_{1}+M_{1}^{\ast})$
is strictly positive definite on the kernel of $M_{0}$ (in \cite{Picard}
this case is called the (P)-degenerate case, since it typically occurs
for parabolic-type problem). Thus, we end up with an equation of the
form 
\begin{equation}
\left(\partial_{0}M_{0}+M_{1}+A\right)u=f,\label{eq:auto_eq}
\end{equation}
whose well-posedness was proved in \cite{Picard} in the case of a
skew-selfadjoint operator $A$. Later on the well-posedness of problems
of the form \prettyref{eq:auto_eq} was shown in the case of $A$
being a maximal monotone operator in \cite{Trostorff2012_NA,Trostorff2012_nonlin_bd}
(for the topic of maximal monotone operators we refer to the monographs
\cite{Brezis,papageogiou,showalter_book}). In \cite{Picard2013_nonauto}
a non-autonomous version of \prettyref{eq:auto_eq} was considered
in the sense that the operators $M_{0}$ and $M_{1}$ were replaced
by operator-valued functions $M_{0},M_{1}:\mathbb{R}\to L(H)$ and
the well-posedness of the corresponding evolutionary problem 
\[
\left(\partial_{0}M_{0}(\cdot)+M_{1}(\cdot)+A\right)u=f
\]
was shown in the case of a skew-selfadjoint operator $A.$ The aim
of this article is to generalize this well-posedness result to the
case of $A$ being a maximal monotone relation, i.e. providing a solution
theory for differential inclusions of the form 
\begin{equation}
(u,f)\in\partial_{0}M_{0}(\cdot)+M_{1}(\cdot)+A.\label{eq:non_auto_inc}
\end{equation}

In the literature we find several approaches to the well-posedness
of non-autonomous differential equations and inclusions and, depending
on the techniques involved, several notions of solutions. For example,
one classical approach, established in a general Banach spaces setting,
is the theory of evolution families introduced by Kato in \cite{Kato1953}
in the case of evolution equations and generalized by Crandall and
Pazy in \cite{Crandall1972} to evolution inclusions. This strategy,
which carries over the idea of semigroup theory to the case of non-autonomous
problems, requires that the differential inclusion is given as a Cauchy-problem,
i.e. an inclusion of the form 
\begin{equation}
(u(t),f(t))\in\partial_{0}+A(t),\label{eq:Cauchy}
\end{equation}
which corresponds to the case of an invertible mapping $M_{0}$ in
our setting. However, in the approach presented here, $M_{0}$ is
allowed to have a non-trivial kernel, which makes the inclusion \prettyref{eq:non_auto_inc}
to be a differential-algebraic problem, which in general may not be
accessible by the theory of semigroups or evolution families in a
straightforward way. Another approach to problems of the form \prettyref{eq:Cauchy}
is to approximate the differential inclusion by difference inclusions,
i.e. one replaces the derivative with respect to time by suitable
difference quotients. The corresponding solutions of the difference
inclusion then uniformly converge to a so-called ``weak'' solution
of \prettyref{eq:Cauchy} (see e.g. \cite{Evans1977,Pavel1981,Kobayasi1984}).
Another notion of solution of differential inclusions of the form
\prettyref{eq:Cauchy} are so-called ``integral solutions'', introduced
by Bénilan \cite{Benilan1972} for autonomous inclusions and generalized
in \cite{Kobayasi1984} to non-autonomous problems, which satisfy
a certain integral inequality. Under suitable assumptions on $A(t)$
one can show that the notion of ``weak'' solutions, i.e. the limit
of solutions of the difference inclusions, coincides with the notion
of ``integral solutions''. 

We emphasize that in all classical approaches, the operator $A$ is
time-dependent. However, looking at concrete examples, in many cases
the coefficients depending on time while the spatial differential
operator is indeed time-independent. Since the coefficients can usually
be incorporated in the operators $M_{0}$ and $M_{1}$, we are led
to assume that these operators depend on time while $A$ is time-independent.
This point of view has the advantage that the time dependent operators
are bounded and thus, we avoid the technicalities arising when dealing
with time-dependent unbounded operators, whose domain may also depend
on time. 

In our approach we consider the operator $\partial_{0}M_{0}(\cdot)+M_{1}(\cdot)+A$
on the right hand side of \prettyref{eq:non_auto_inc} as an object
in time and space. More precisely, the operators involved are defined
on an exponentially weighted $L_{2}$-space of Hilbert-space valued
functions, and we are seeking for solutions $u$ of \prettyref{eq:non_auto_inc}
in the sense that 
\[
(u,f)\in\overline{\partial_{0}M_{0}(\cdot)+M_{1}(\cdot)+A},
\]
where the closure is taken with respect to the topology on this $L_{2}$-space
of Hilbert space valued functions. This can be seen as an $L_{2}$-
analogue to the notion of weak solutions due to Brezis (see \cite[Definition 3.1]{Brezis}).
Thus, the well-posedness of \prettyref{eq:non_auto_inc} relies on
the invertibility of $\overline{\partial_{0}M_{0}(\cdot)+M_{1}(\cdot)+A}$,
which will be shown by proving that $\overline{\partial_{0}M_{0}(\cdot)+M_{1}(\cdot)+A}-c$
defines a maximal monotone operator in time and space for some $c>0$.
For doing so, we establish the time derivative $\partial_{0}$ in
an exponentially weighted $L_{2}$-space in order to obtain a normal,
boundedly invertible operator (cf. \cite{picard1989hilbert,Picard_McGhee}).
The operator $\partial_{0}M_{0}(\cdot)+M_{1}(\cdot)$ turns out to
be strictly maximal monotone as an operator in time and space and
hence, the well-posedness of \prettyref{eq:non_auto_inc} can be shown
by applying well-known perturbation results for maximal monotone relations. 

In addition to the well-posedness, we address the question of causality
(see e.g. \cite{lakshmikantham2010theory} or \cite{Waurick2013_causality}
for an alternative definition), which is a characteristic property
for processes evolving in time. Roughly speaking, causality means
that the behavior of the solution $u$ of \prettyref{eq:non_auto_inc}
should not depend on the future behavior of the given right hand side
(for the exact definition in our framework see Definition \ref{DefCausal}). 

The article is structured as follows. In Section 2 we recall the definition
of the time derivative $\partial_{0}$ and some basic facts on maximal
monotone operators in Hilbert spaces. Section 3 is devoted to the
proof of our main theorem (\prettyref{thm:sol-theory}), stating the
well-posedness result for problems of the form \prettyref{eq:non_auto_inc}
and the causality of the corresponding solution operator. In the concluding
section we apply our results to two examples from the theory of plasticity.
The first one deals with a non-autonomous version of the equations
of thermoplasticity, where the inelastic part of the strain and the
stress are coupled by an differential inclusion. In the second one
we consider the non-autonomous equations of viscoplasticity, where
the inelastic strain is given in terms of an internal variable (for
constitutive equations with internal variables we refer to the monograph
\cite{Alber_1998}).

In the following let $H$ be a complex Hilbert space with inner product
$\langle\cdot|\cdot\rangle$, assumed to be linear in the second and
conjugate linear in the first argument and we denote by $|\cdot|$
the induced norm.

\section{Preliminaries}

\subsection{The time derivative}

Following the strategy in \cite{Picard_McGhee}, we introduce the
derivative as a normal, boundedly invertible operator in an exponentially
weighted $L_{2}$-space. For the proofs of the forthcoming statements
we refer to \cite{Kalauch2011,Picard_McGhee}. For $\rho\in\mathbb{R}$
we define the space $H_{\rho,0}(\mathbb{R};H)$ as the completion
of $C_{c}^{\infty}(\mathbb{R};H)$ -- the space of arbitrarily often
differentiable functions with compact support in $\mathbb{R}$ taking
values in $H$ -- with respect to the norm induced by the inner product
\[
\langle\phi|\psi\rangle_{H_{\rho,0}(\mathbb{R};H)}\coloneqq\intop_{\mathbb{R}}\langle\phi(t)|\psi(t)\rangle e^{-2\rho t}\mbox{ d}t\quad(\phi,\psi\in C_{c}^{\infty}(\mathbb{R};H)).
\]
Note that in the case $\rho=0$, this is just the usual $L_{2}$-space
of (equivalence classes of) square integrable functions with values
in $H$, i.e. $H_{0,0}(\mathbb{R};H)=L_{2}(\mathbb{R};H).$ On the
Hilbert space $H_{\rho,0}(\mathbb{R};H)$ we define the derivative
$\partial_{0,\rho}$ as the closure of the linear operator 
\begin{align*}
C_{c}^{\infty}(\mathbb{R};H)\subseteq H_{\rho,0}(\mathbb{R};H) & \to H_{\rho,0}(\mathbb{R};H)\\
\phi & \mapsto\phi'.
\end{align*}
Then $\partial_{0,\rho}$ is a normal operator with $\partial_{0,\rho}^{\ast}=-\partial_{0,\rho}+2\rho$
and consequently $\Re\partial_{0,\rho}=\rho$. In the case $\rho=0$
this operator coincides with the usual weak derivative on $L_{2}(\mathbb{R};H)$
with domain $H^{1}(\mathbb{R};H)=W_{2}^{1}(\mathbb{R};H).$ For $\rho\ne0$
the operator $\partial_{0,\rho}$ has a bounded inverse with $\|\partial_{0,\rho}^{-1}\|_{L(H_{\rho,0}(\mathbb{R};H))}=\frac{1}{|\rho|}$
(see \cite[Corollary 2.5]{Kalauch2011}). More precisely, the inverse
is given by 
\begin{align}
\left(\partial_{0,\rho}^{-1}u\right)(t) & =\begin{cases}
\intop_{-\infty}^{t}u(s)\mbox{ d}s & \mbox{ if }\rho>0,\\
-\intop_{t}^{\infty}u(s)\mbox{ d}s & \mbox{ if }\rho<0
\end{cases}\quad(u\in H_{\rho,0}(\mathbb{R};H),\, t\in\mathbb{R}\mbox{ a.e.}).\label{eq:inverse_d0}
\end{align}
Since we are interested in the forward causal case (see Definition
\ref{DefCausal} below), throughout we may assume that $\rho>0.$
Next, we state an approximation result for elements in the domain
of $\partial_{0,\rho}.$ For this we denote by $\tau_{h}$ for $h\in\mathbb{R}$
the translation operator on $H_{\rho,0}(\mathbb{R};H)$ given by $\left(\tau_{h}u\right)(t)\coloneqq u(t+h)$
for $u\in H_{\rho,0}(\mathbb{R};H)$ and almost every $t\in\mathbb{R}.$
Moreover we define the Hilbert space $H_{\rho,1}(\mathbb{R};H)$ as
the domain of $\partial_{0,\rho}$ equipped with the inner product
\[
\langle u|v\rangle_{H_{\rho,1}(\mathbb{R};H)}\coloneqq\langle\partial_{0,\rho}u|\partial_{0,\rho}v\rangle_{H_{\rho,0}(\mathbb{R};H)}\quad(u,v\in D(\partial_{0,\rho})).
\]

\begin{prop}
\label{prop:chara_diff}Let $u\in H_{\rho,0}(\mathbb{R};H).$ Then
$u\in H_{\rho,1}(\mathbb{R};H)$ if and only if the set of difference
quotients $\left\{ \left.\frac{1}{h}\left(\tau_{h}u-u\right)\,\right|\, h\in]0,t]\right\} $
is bounded in $H_{\rho,0}(\mathbb{R};H)$ for some $t>0.$ Moreover,
for $u\in H_{\rho,1}(\mathbb{R};H)$ we have 
\[
\frac{1}{h}\left(\tau_{h}u-u\right)\to\partial_{0,\rho}u\mbox{ in }H_{\rho,0}(\mathbb{R};H)\mbox{ as }h\to0+.
\]
\end{prop}
\begin{proof}
For $h>0$ we define the operator 
\begin{align*}
D_{h}:H_{\rho,1}(\mathbb{R};H) & \to H_{\rho,0}(\mathbb{R};H)\\
u & \mapsto\frac{1}{h}\left(\tau_{h}u-u\right).
\end{align*}
Obviously, this operator is linear and we estimate 
\begin{align*}
|D_{h}u|_{H_{\rho,0}(\mathbb{R};H)}^{2} & =\intop_{\mathbb{R}}\left|\frac{1}{h}\left(u(t+h)-u(t)\right)\right|^{2}e^{-2\rho t}\mbox{ d}t\\
 & =\intop_{\mathbb{R}}\frac{1}{h^{2}}\left|\intop_{0}^{h}\partial_{0,\rho}u(t+s)\mbox{ d}s\right|^{2}e^{-2\rho t}\mbox{ d}t\\
 & \leq\intop_{\mathbb{R}}\frac{1}{h}\intop_{0}^{h}|\partial_{0,\rho}u(t+s)|^{2}\mbox{ d}s\, e^{-2\rho t}\mbox{ d}t\\
 & =\frac{1}{h}\intop_{0}^{h}\intop_{\mathbb{R}}|\partial_{0,\rho}u(t+s)|^{2}e^{-2\rho t}\mbox{ d}t\mbox{ d}s\\
 & \leq e^{2\rho h}|u|_{H_{\rho,1}(\mathbb{R};H)}^{2}
\end{align*}
for each $u\in H_{\rho,1}(\mathbb{R};H).$ Thus, the family $\left(D_{h}\right)_{h\in]0,t]}$
is uniformly bounded in the space $L(H_{\rho,1}(\mathbb{R};H),H_{\rho,0}(\mathbb{R};H))$
for every $t>0$. Moreover, for $\phi\in C_{c}^{\infty}(\mathbb{R};H)$
we have 
\[
|D_{h}\phi-\partial_{0,\rho}\phi|_{H_{\rho,0}(\mathbb{R};H)}\to0\quad(h\to0+)
\]
by the dominated convergence theorem. Since $C_{c}^{\infty}(\mathbb{R};H)$
is dense in $H_{\rho,1}(\mathbb{R};H)$ we derive that 
\[
D_{h}u\to\partial_{0,\rho}u\mbox{ in }H_{\rho,0}(\mathbb{R};H)\mbox{ as }h\to0+,
\]
if $u\in H_{\rho,1}(\mathbb{R};H)$. Assume now that $\left\{ \left.\frac{1}{h}\left(\tau_{h}u-u\right)\,\right|\, h\in]0,t]\right\} $
is bounded in $H_{\rho,0}(\mathbb{R};H)$ for some $t>0.$ Then we
can choose a sequence $\left(h_{n}\right)_{n\in\mathbb{N}}$ in $]0,t]$
such that $h_{n}\to0$ as $n\to\infty$ and $\left(\frac{1}{h_{n}}\left(\tau_{h_{n}}u-u\right)\right)_{n\in\mathbb{N}}$
is weakly convergent. We denote its weak limit by $w\in H_{\rho,0}(\mathbb{R};H).$
Then we compute for $\phi\in C_{c}^{\infty}(\mathbb{R};H)$
\begin{align*}
\langle w|\phi\rangle_{H_{\rho,0}(\mathbb{R};H)} & =\lim_{n\to\infty}\intop_{\mathbb{R}}\frac{1}{h_{n}}\langle u(t+h_{n})-u(t)|\phi(t)\rangle e^{-2\rho t}\mbox{ d}t\\
 & =\lim_{n\to\infty}\intop_{\mathbb{R}}\left\langle u(t)\left|\frac{1}{h_{n}}\left(\phi(t-h_{n})e^{2\rho h_{n}}-\phi(t)\right)\right.\right\rangle e^{-2\rho t}\mbox{ d}t\\
 & =\intop_{\mathbb{R}}\langle u(t)|-\phi'(t)+2\rho\phi(t)\rangle e^{-2\rho t}\mbox{ d}t,
\end{align*}
by the dominated convergence theorem. Thus, we have for all $\phi\in C_{c}^{\infty}(\mathbb{R};H)$
\[
\langle w|\phi\rangle_{H_{\rho,0}(\mathbb{R};H)}=\langle u|\partial_{0,\rho}^{\ast}\phi\rangle_{H_{\rho,0}(\mathbb{R};H)}.
\]
Since $C_{c}^{\infty}(\mathbb{R};H)$ is a core for $\partial_{0,\rho}^{\ast}$
we obtain $u\in H_{\rho,1}(\mathbb{R};H).$
\end{proof}

\subsection{Maximal monotone relations}

In this section we recall some basic results on maximal monotone operators.
Instead of considering the operators as set-valued mappings, we prefer
to use the notion of binary relations. The proofs of the results can
be found for instance in the monographs \cite{Brezis,papageogiou}. 
\begin{defn}
A (binary) relation $A\subseteq H\oplus H$ is called \emph{monotone},
if for all pairs $(u,v),(x,y)\in A$ the inequality 
\[
\Re\langle u-x|v-y\rangle\geq0
\]
holds. Moreover, $A$ is called \emph{maximal monotone, }if $A$ is
monotone and there exists no proper monotone extension of $A$, i.e.
for every monotone $B\subseteq H\oplus H$ with $A\subseteq B$ it
follows that $A=B.$
\end{defn}
In order to deal with relations we fix some notation, which will be
used in the forthcoming sections.
\begin{defn}
For two relations $A,B\subseteq H\oplus H$ and $\lambda\in\mathbb{C}$
we define the relation $\lambda A+B$ by 
\[
\lambda A+B\coloneqq\left\{ (x,\lambda y+z)\,|\,(x,y)\in A,(x,z)\in B\right\} .
\]
The inverse relation $A^{-1}$ is given by 
\[
A^{-1}\coloneqq\left\{ (y,x)\,|\,(x,y)\in A\right\} .
\]
Furthermore, for a subset $M\subseteq H$ we define the \emph{pre-set}
of $M$ under $A$ by 
\[
[M]A\coloneqq\left\{ x\in H\,|\,\exists y\in M:(x,y)\in A\right\} 
\]
and the \emph{post-set}%
\footnote{These notions generalize the well-known concepts of pre-image and
image in the case of mappings. However, since it seems to be inappropriate
to speak of images in case of a relation, we choose the notions pre-
and post-set.%
} of $M$ under $A$ by 
\[
A[M]\coloneqq\left\{ y\in H\,|\,\exists x\in M:(x,y)\in A\right\} .
\]

Moreover, $A$ is called \emph{bounded,} if for every bounded set
$M\subseteq H$ the post-set $A[M]$ is bounded.
\end{defn}
In 1962 G. Minty proved the following characterization of maximal
monotonicity.
\begin{thm}[G. Minty, \cite{Minty}]
\label{thm:Minty} Let $A\subseteq H\oplus H$ be a monotone relation.
Then the following statements are equivalent:

\begin{enumerate}[(i)]

\item $A$ is maximal monotone,

\item there exists $\lambda>0$ such that%
\footnote{We indicate the identity on $H$ by $1$.%
} $(1+\lambda A)[H]=H,$ 

\item for every $\lambda>0$ we have $(1+\lambda A)[H]=H.$

\end{enumerate}
\end{thm}
In order to formulate inclusions of the form \prettyref{eq:non_auto_inc}
we need to extend a relation $A\subseteq H\oplus H$ to a relation
on $H_{\rho,0}(\mathbb{R};H)$ for $\rho>0.$ This is done by setting
\begin{equation}
A_{\rho}\coloneqq\left\{ (u,v)\in H_{\rho,0}(\mathbb{R};H)\oplus H_{\rho,0}(\mathbb{R};H)\,|\,(u(t),v(t))\in A\mbox{ for almost every }t\in\mathbb{R}\right\} .\label{eq:extension}
\end{equation}
The relation $A_{\rho}$ then interchanges with the translation operator
$\tau_{h}$ for $h\in\mathbb{R}$ in the sense that 
\begin{equation}
(u,v)\in A_{\rho}\Leftrightarrow\forall h\in\mathbb{R}:(\tau_{h}u,\tau_{h}v)\in A_{\rho}.\label{eq:extension_autonomous}
\end{equation}
We recall the following result from \cite{Brezis} on extensions of
maximal monotone relations.
\begin{lem}[{\cite[Exemple 2.3.3]{Brezis}}]
\label{lem:extension_max_mon}Let $A\subseteq H\oplus H$ be maximal
monotone and $\rho\geq0$. If $(0,0)\in A$, then $A_{\rho}$ is maximal
monotone, too.
\end{lem}
If $A\subseteq H\oplus H$ is a monotone relation it follows that
$(1+\lambda A)^{-1}$ is a Lipschitz-continuous mapping with Lipschitz-constant
less than or equal to $1$ for each $\lambda>0$. Furthermore, if
$A$ is maximal monotone, the mapping $(1+\lambda A)^{-1}$ is defined
on the whole space $H$ by \prettyref{thm:Minty}. 
\begin{defn}
Let $A\subseteq H\oplus H$ be maximal monotone and $\lambda>0.$
The \emph{Yosida approximation }$A_{\lambda}$ of $A$ is defined
by 
\[
A_{\lambda}\coloneqq\lambda^{-1}\left(1-(1+\lambda A)^{-1}\right).
\]
The mapping $A_{\lambda}$ is monotone and Lipschitz-continuous with
a Lipschitz-constant less than or equal to $\lambda^{-1}$ (see \cite[Proposition 2.6]{Brezis}).
\end{defn}
We close this subsection by stating some perturbation results for
maximal monotone relations, which provide the key argument for our
proof of well-posedness of evolutionary inclusions of the form \prettyref{eq:non_auto_inc}. 
\begin{prop}
\label{prop:pert-Lip1}Let $A\subseteq H\oplus H$ be maximal monotone
and $B:H\to H$ be Lipschitz-continuous. Furthermore, let $A+B$ be
monotone. Then $A+B$ is maximal monotone.\end{prop}
\begin{proof}
If $B$ is constant the assertion holds trivially. Assume now $B$
is not constant. By \prettyref{thm:Minty} it suffices to check that
there exists $\lambda>0$ such that $\left(1+\lambda(A+B)\right)[H]=H$.
Let%
\footnote{For a Lipschitz-continuous mapping $F:X\to Y$ between two metric
spaces $X$ and $Y$, we denote by $|F|_{\mathrm{Lip}}$ the smallest
Lipschitz-constant of $F$, i.e. 
\[
|F|_{\mathrm{Lip}}\coloneqq\inf\left\{ L>0\,|\,\forall x,y\in X:d_{Y}(F(x),F(y))\leq Ld_{X}(x,y)\right\} .
\]
} $0<\lambda<|B|_{\mathrm{Lip}}^{-1}$ and $y\in H$. Then by the contraction
mapping theorem there exists a fixed point $x\in H$ of the mapping
\[
H\ni u\mapsto(1+\lambda A)^{-1}(y-\lambda B(u)).
\]
This fixed point satisfies 
\[
(x,y)\in1+\lambda(A+B).\tag*{\qedhere}
\]
\end{proof}
\begin{cor}
\label{cor:Lipschitz_maxmon}Let $B:H\to H$ be monotone and Lipschitz-continuous.
Then $B$ is maximal monotone.\end{cor}
\begin{proof}
This follows from \prettyref{prop:pert-Lip1} with $A=0.$ \end{proof}
\begin{cor}[{\cite[Lemme 2.4]{Brezis}}]
\label{cor:pert_Lip} Let $A\subseteq H\oplus H$ be maximal monotone
and $B:H\to H$ be a monotone, Lipschitz-continuous mapping. Then
$A+B$ is maximal monotone.\end{cor}
\begin{proof}
The statement follows from \prettyref{prop:pert-Lip1} since the sum
of two monotone relations is again monotone.
\end{proof}
Let $A,B\subseteq H\oplus H$ be maximal monotone. Then, by \prettyref{cor:pert_Lip},
the relation $A+B_{\lambda}$ is maximal monotone for each $\lambda>0$
and thus, for $y\in H$ there exists a unique $x_{\lambda}\in H$
such that 
\[
(x_{\lambda},y)\in1+A+B_{\lambda},
\]
according to Minty's theorem (\prettyref{thm:Minty}). Using this
observation, one can show the following perturbation result.
\begin{prop}[{\cite[Proposition 3.1]{papageogiou}}]
\label{prop:pert} Let $A,B\subseteq H\oplus H$ be maximal monotone
with $[H]A\cap[H]B\ne\emptyset$ and $y\in H.$ Moreover, for $\lambda>0$
let $x_{\lambda}\in H$ such that $(x_{\lambda},y)\in1+A+B_{\lambda}.$
Then, there exists $x\in H$ with $(x,y)\in1+A+B$ if and only if
the family $\left(B_{\lambda}(x_{\lambda})\right)_{\lambda\in]0,\infty[}$
is bounded. \end{prop}
\begin{cor}[{\cite[Proposition 1.22]{Trostorff_2011}}]
\label{cor:bounded_pert}  Let $A,B\subseteq H\oplus H$ be maximal
monotone with $[H]A\cap[H]B\ne\emptyset.$ Moreover, assume that $B$
is bounded. Then $A+B$ is maximal monotone.
\end{cor}

\section{Solution theory}

In this section we provide a solution theory for differential inclusions
of the form \prettyref{eq:non_auto_inc}. More precisely we show that
the problem 
\begin{equation}
(u,f)\in\overline{\partial_{0,\rho}M_{0}(\mathrm{m})+M_{1}(\mathrm{m})+A_{\rho}}\label{eq:prob}
\end{equation}
is well-posed in $H_{\rho,0}(\mathbb{R};H)$ for sufficiently large
$\rho>0$ and that the corresponding solution operator 
\[
\left(\overline{\partial_{0,\rho}M_{0}(\mathrm{m})+M_{1}(\mathrm{m})+A_{\rho}}\right)^{-1}:H_{\rho,0}(\mathbb{R};H)\to H_{\rho,0}(\mathbb{R};H)
\]
is causal. Throughout let $A\subseteq H\oplus H$ be maximal monotone
with $(0,0)\in A$ and $M_{0},M_{1}:\mathbb{R}\to L(H)$ be strongly
measurable and bounded functions. Then we denote the associated multiplication
operators on $H_{\rho,0}(\mathbb{R};H)$ by $M_{0}(\mathrm{m})$ and
$M_{1}(\mathrm{m}),$ respectively, where $\mathrm{m}$ serves as
a reminder for the ``multiplication by the argument'', i.e. $M_{0}(\mathrm{m}):H_{\rho,0}(\mathbb{R};H)\to H_{\rho,0}(\mathbb{R};H)$
with $\left(M_{0}(\mathrm{m})u\right)(t)\coloneqq M_{0}(t)u(t)$ for
every $u\in H_{\rho,0}(\mathbb{R};H)$ and almost every $t\in\mathbb{R}$
and analogously for $M_{1}(\mathrm{m})$%
\footnote{Note that due to the boundedness of $M_{0}$ and $M_{1}$ the operators
$M_{0}(\mathrm{m})$ and $M_{1}(\mathrm{m})$ are bounded on $H_{\rho,0}(\mathbb{R};H)$
for each $\rho>0$ with $\|M_{0}(\mathrm{m})\|_{L(H)}\leq|M_{0}|_{\infty}$
and $\|M_{1}(\mathrm{m})\|_{L(H)}\leq|M_{1}|_{\infty}$. %
}. Following \cite{Picard2013_nonauto} we require that the pair $(M_{0},M_{1})$
satisfies the following conditions.

\begin{cond}$\,$

\begin{enumerate}[(a)]

\item $M_{0}$ is Lipschitz-continuous and for every $t\in\mathbb{R}$
the operator $M_{0}(t)$ is selfadjoint. 

\item There exists a set $N\subseteq\mathbb{R}$ of Lebesgue-measure
$0$, such that for every $x\in H$ the mapping 
\[
\mathbb{R}\setminus N\ni t\mapsto M_{0}(t)x
\]
is differentiable%
\footnote{If $H$ is separable, this assumptions already follows by the Lipschitz-continuity
of $M_{0}$.%
}. 

\item  The kernel of $M_{0}(t)$ is independent of $t\in\mathbb{R},$
i.e. $N(M_{0})\coloneqq[\{0\}]M_{0}(0)=[\{0\}]M_{0}(t)$ for all $t\in\mathbb{R}.$ 

\end{enumerate}

We denote by $\iota_{0}:N(M_{0})\to H$ the canonical embedding of
$N(M_{0})$ into $H$. Then an easy computation shows that $\iota_{0}\iota_{0}^{\ast}:H\to H$
is the orthogonal projector onto the closed subspace $N(M_{0})$ (see
e.g. \cite[Lemma 3.2]{Picard2013_fractional}). In the same way we
denote by $\iota_{1}:N(M_{0})^{\bot}\to H$ the canonical embedding
of $N(M_{0})^{\bot}=\overline{M_{0}(t)[H]}$ into $H$. 

Finally, we require that

\begin{enumerate}[(d)]

\item There exist $c_{0},c_{1}>0$ such that for all $t\in\mathbb{R}$
the operators $\iota_{1}^{\ast}M_{0}(t)\iota_{1}-c_{0}$ and $\iota_{0}^{\ast}\Re M_{1}(t)\iota_{0}-c_{1}$
are monotone in $N(M_{0})^{\bot}$ and $N(M_{0})$, respectively. 

\end{enumerate} 

\end{cond}

Note that in \cite{Picard2013_nonauto} condition (c) is not required
and (d) is replaced by a more general monotonicity constraint. However,
in order to apply perturbation results, which will be a key tool for
proving the well-posedness of \prettyref{eq:prob}, we need to impose
the constraints (c) and (d) above (compare \cite[Theorem 2.19]{Picard2013_nonauto}). 
\begin{rem}
Note that under the conditions above, the operators $\iota_{1}^{\ast}M_{0}(\mathrm{m})\iota_{1}$
and $\iota_{0}^{\ast}M_{1}(\mathrm{m})\iota_{0}$ are continuously
invertible as operators in $L(H_{\rho,0}(\mathbb{R};N(M_{0})^{\bot}))$
and $L(H_{\rho,0}(\mathbb{R};N(M_{0}))$), respectively. 
\end{rem}
We recall the following result from \cite{Picard2013_nonauto}.
\begin{lem}[{\cite[Lemma 2.1]{Picard2013_nonauto}}]
 For $t\in\mathbb{R}$ the mapping 
\begin{align*}
M_{0}'(t):H & \to H\\
x & \mapsto\begin{cases}
\left(M_{0}(\cdot)x\right)'(t) & \mbox{ if }t\in\mathbb{R}\setminus N,\\
0 & \mbox{ otherwise}
\end{cases}
\end{align*}
is a bounded linear operator with $\|M_{0}'(t)\|_{L(H)}\leq|M_{0}|_{\mathrm{Lip}}$
and thus, gives rise to a bounded multiplication operator $M_{0}'(\mathrm{m})\in L(H_{\rho,0}(\mathbb{R};H))$
for each $\rho\in]0,\infty[$. Furthermore, $M_{0}'(\mathrm{m})$
is selfadjoint. Moreover, for $u\in H_{\rho,1}(\mathbb{R};H)$ the
\emph{chain rule
\begin{equation}
\partial_{0,\rho}M_{0}(\mathrm{m})u=M_{0}(\mathrm{m})\partial_{0,\rho}u+M_{0}'(\mathrm{m})u\label{eq:chainrule}
\end{equation}
}holds.\end{lem}
\begin{rem}
From \prettyref{eq:chainrule} we derive 
\begin{align}
\partial_{0,\rho}\left(\iota_{1}^{\ast}M_{0}(\mathrm{m})\iota_{1}\right)u & =\iota_{1}^{\ast}\partial_{0,\rho}M_{0}(\mathrm{m})\iota_{1}u\nonumber \\
 & =\iota_{1}^{\ast}M_{0}(\mathrm{m})\partial_{0,\rho}\iota_{1}u+\iota_{1}^{\ast}M_{0}'(\mathrm{m})\iota_{1}u\nonumber \\
 & =\left(\iota_{1}^{\ast}M_{0}(\mathrm{m})\iota_{1}\right)\partial_{0,\rho}u+\left(\iota_{1}^{\ast}M_{0}'(\mathrm{m})\iota_{1}\right)u\label{eq:chainrule_modified}
\end{align}
for $u\in H_{\rho,1}(\mathbb{R};N(M_{0})^{\bot}).$
\end{rem}
In the following two subsections we prove our main theorem.
\begin{thm}[Solution Theory]
\label{thm:sol-theory} Let $(M_{0},M_{1})$ be a pair of $L(H)$-valued
strongly measurable functions satisfying (a)-(d). Moreover, let $A\subseteq H\oplus H$
be a maximal monotone relation with $(0,0)\in A.$ Then there exists
$\rho_{0}>0$ such that for every $\rho\geq\rho_{0}$
\[
\left(\overline{\partial_{0,\rho}M_{0}(\mathrm{m})+M_{1}(\mathrm{m})+A_{\rho}}\right)^{-1}:H_{\rho,0}(\mathbb{R};H)\to H_{\rho,0}(\mathbb{R};H)
\]
is a Lipschitz-continuous, causal mapping. Moreover, the mapping is
independent of $\rho$ in the sense that, for $\nu,\rho\geq\rho_{0}$
and $f\in H_{\rho,0}(\mathbb{R};H)\cap H_{\nu,0}(\mathbb{R};H)$ 
\[
\left(\overline{\partial_{0,\rho}M_{0}(\mathrm{m})+M_{1}(\mathrm{m})+A_{\rho}}\right)^{-1}(f)=\left(\overline{\partial_{0,\nu}M_{0}(\mathrm{m})+M_{1}(\mathrm{m})+A_{\nu}}\right)^{-1}(f)
\]
as functions in $L_{2,\mathrm{loc}}(\mathbb{R};H).$
\end{thm}

\subsection{Well-posedness}

We begin with characterizing the elements belonging to the domain
of $\partial_{0,\rho}M_{0}(\mathrm{m})$.
\begin{lem}
\label{lem:char_dom}Let $\rho>0$. Then an element $u\in H_{\rho,0}(\mathbb{R};H)$
belongs to $D(\partial_{0,\rho}M_{0}(\mathrm{m}))$ if and only if
$\iota_{1}^{\ast}u\in H_{\rho,1}(\mathbb{R};N(M_{0})^{\bot})$.\end{lem}
\begin{proof}
Let $u\in H_{\rho,0}(\mathbb{R};H).$ First we assume that $u\in D(\partial_{0,\rho}M_{0}(\mathrm{m})).$
Then for $\phi\in C_{c}^{\infty}(\mathbb{R};N(M_{0})^{\bot})$ we
compute, using the continuous invertibility of $\iota_{1}^{\ast}M_{0}(\mathrm{m})\iota_{1}$
and the chain rule \prettyref{eq:chainrule_modified} 
\begin{align*}
 & \langle\iota_{1}^{\ast}u|\partial_{0,\rho}\phi\rangle_{H_{\rho,0}(\mathbb{R};N(M_{0})^{\bot})}\\
 & =\left\langle \left(\iota_{1}^{\ast}M_{0}(\mathrm{m})\iota_{1}\right)^{-1}\iota_{1}^{\ast}M_{0}(\mathrm{m})\iota_{1}\iota_{1}^{\ast}u\left|\partial_{0,\rho}\phi\right.\right\rangle _{H_{\rho,0}(\mathbb{R};N(M_{0})^{\bot})}\\
 & =\left\langle \iota_{1}^{\ast}M_{0}(\mathrm{m})u\left|\left(\iota_{1}^{\ast}M_{0}(\mathrm{m})\iota_{1}\right)^{-1}\partial_{0,\rho}\phi\right.\right\rangle _{H_{\rho,0}(\mathbb{R};N(M_{0})^{\bot})}\\
 & =\left\langle \iota_{1}^{\ast}M_{0}(\mathrm{m})u\left|\partial_{0,\rho}\left(\iota_{1}^{\ast}M_{0}(\mathrm{m})\iota_{1}\right)^{-1}\phi\right.\right\rangle _{H_{\rho,0}(\mathbb{R};N(M_{0})^{\bot})}\\
 & \quad+\left\langle \iota_{1}^{\ast}M_{0}(\mathrm{m})u\left|\left(\iota_{1}^{\ast}M_{0}(\mathrm{m})\iota_{1}\right)^{-1}\left(\iota_{1}^{\ast}M_{0}'(\mathrm{m})\iota_{1}\right)\left(\iota_{1}^{\ast}M_{0}(\mathrm{m})\iota_{1}\right)^{-1}\phi\right.\right\rangle _{H_{\rho,0}(\mathbb{R};N(M_{0})^{\bot})}\\
 & =\left\langle \left.\left(\iota_{1}^{\ast}M_{0}(\mathrm{m})\iota_{1}\right)^{-1}\left(\partial_{0,\rho}^{\ast}\iota_{1}^{\ast}M_{0}(\mathrm{m})u+\left(\iota_{1}^{\ast}M_{0}'(\mathrm{m})\iota_{1}\right)\iota_{1}^{\ast}u\right)\right|\phi\right\rangle _{H_{\rho,0}(\mathbb{R};N(M_{0})^{\bot})}.
\end{align*}
This proves that $\iota_{1}^{\ast}u\in D(\partial_{0,\rho}^{\ast})=D(\partial_{0,\rho})$,
since $C_{c}^{\infty}(\mathbb{R};N(M_{0})^{\bot})$ is dense in $H_{\rho,1}(\mathbb{R};N(M_{0})^{\bot}).$
On the other hand if $\iota_{1}^{\ast}u\in D(\partial_{0,\rho})$
the assertion follows by the chain rule \prettyref{eq:chainrule_modified}.
\end{proof}
As it was done in \cite{Trostorff2012_NA} in the autonomous case
we prove the strict maximal monotonicity of the operator $\partial_{0,\rho}M_{0}(\mathrm{m})+M_{1}(\mathrm{m})$
for sufficiently large $\rho.$ 
\begin{lem}
\label{lem:max_mon_material-1} Let $\rho>0$ and $a\in\mathbb{R}$.
Then for each $u\in D(\partial_{0,\rho}M_{0}(\mathrm{m}))$ and $\varepsilon>0$
the estimate 
\begin{multline*}
\Re\int\limits _{-\infty}^{a}\langle(\partial_{0,\rho}M_{0}(\mathrm{m})+M_{1}(\mathrm{m}))u(t)|u(t)\rangle e^{-2\rho t}\,\mathrm{d}t\\
\geq\left(\rho c_{0}-\frac{1}{2}|M_{0}|_{\mathrm{Lip}}-|M_{1}|_{\infty}-\frac{1}{\varepsilon}|M_{1}|_{\infty}^{2}\right)\int\limits _{-\infty}^{a}|\iota_{1}^{\ast}u(t)|^{2}e^{-2\rho t}\,\mathrm{d}t+(c_{1}-\varepsilon)\int\limits _{-\infty}^{a}|\iota_{0}^{\ast}u(t)|^{2}e^{-2\rho t}\,\mathrm{d}t
\end{multline*}
holds. In particular, by letting $a$ tend to infinity, we have 
\begin{multline}
\Re\langle(\partial_{0,\rho}M_{0}(\mathrm{m})+M_{1}(\mathrm{m}))u|u\rangle_{H_{\rho,0}(\mathbb{R};H)}\\
\geq\left(\rho c_{0}-\frac{1}{2}|M_{0}|_{\mathrm{Lip}}-|M_{1}|_{\infty}-\frac{1}{\varepsilon}|M_{1}|_{\infty}^{2}\right)|\iota_{1}^{\ast}u|_{H_{\rho,0}(\mathbb{R};N(M_{0})^{\bot})}^{2}+(c_{1}-\varepsilon)|\iota_{0}^{\ast}u|_{H_{\rho,0}(\mathbb{R};N(M_{0}))}^{2}.\label{eq:pos_def_mat_law}
\end{multline}
Moreover, for each $0<\tilde{c}<c_{1}$ there exists $\rho_{0}>0$
such that for all $\rho\geq\rho_{0}$ the operator $\partial_{0,\rho}M_{0}(\mathrm{m})+M_{1}(\mathrm{m})-\tilde{c}$
is maximal monotone.\end{lem}
\begin{proof}
Let $\phi\in C_{c}^{\infty}(\mathbb{R};N(M_{0})^{\bot})$. Then we
compute
\begin{align*}
 & \Re\int\limits _{-\infty}^{a}\langle(\partial_{0,\rho}\iota_{1}^{\ast}M_{0}(\mathrm{m})\iota_{1})\phi(t)|\phi(t)\rangle e^{-2\rho t}\mbox{ d}t\\
 & =\frac{1}{2}\int\limits _{-\infty}^{a}(\langle(\partial_{0,\rho}\iota_{1}^{\ast}M_{0}(\mathrm{m})\iota_{1})\phi(t)|\phi(t)\rangle+\langle\phi(t)|(\partial_{0,\rho}\iota_{1}^{\ast}M_{0}(\mathrm{m})\iota_{1})\phi(t)\rangle)e^{-2\rho t}\mbox{ d}t\\
 & =\frac{1}{2}\int\limits _{-\infty}^{a}(\langle(\partial_{0,\rho}\iota_{1}^{\ast}M_{0}(\mathrm{m})\iota_{1})\phi(t)|\phi(t)\rangle+\langle\phi(t)|\iota_{1}^{\ast}M_{0}(t)\iota_{1}(\partial_{0,\rho}\phi)(t)+\iota_{1}^{\ast}M_{0}'(t)\iota_{1}\phi(t)\rangle)e^{-2\rho t}\mbox{ d}t\\
 & =\frac{1}{2}\int\limits _{-\infty}^{a}\langle\iota_{1}^{\ast}M_{0}(\cdot)\iota_{1}\phi(\cdot)|\phi(\cdot)\rangle'(t)e^{-2\rho t}\mbox{ d}t+\frac{1}{2}\int\limits _{-\infty}^{a}\langle\phi(t)|\iota_{1}^{\ast}M_{0}'(t)\iota_{1}\phi(t)\rangle e^{-2\rho t}\mbox{ d}t\\
 & \geq\frac{1}{2}\langle\iota_{1}^{\ast}M_{0}(a)\iota_{1}\phi(a)|\phi(a)\rangle e^{-2\rho a}+\rho\int\limits _{-\infty}^{a}\langle\iota_{1}^{\ast}M_{0}(t)\iota_{1}\phi(t)|\phi(t)\rangle e^{-2\rho t}\mbox{ d}t-\frac{1}{2}|M_{0}|_{\mathrm{Lip}}\int\limits _{-\infty}^{a}|\phi(t)|^{2}e^{-2\rho t}\mbox{ d}t\\
 & \geq(\rho c_{0}-\frac{1}{2}|M_{0}|_{\mathrm{Lip}})\int\limits _{-\infty}^{a}|\phi(t)|^{2}e^{-2\rho t}\mbox{ d}t.
\end{align*}
Since $C_{c}^{\infty}(\mathbb{R};N(M_{0})^{\bot})$ is dense in $H_{\rho,1}(\mathbb{R};N(M_{0})^{\bot})$
and%
\footnote{For a function $g\in L_{\infty}(\mathbb{R})$ we denote by $g(\mathrm{m})$
the corresponding multiplication operator on $H_{\rho,0}(\mathbb{R};H)$,
i.e. $\left(g(\mathrm{m})u\right)(t)\coloneqq g(t)u(t)$ for $u\in H_{\rho,0}(\mathbb{R};H)$
and almost every $t\in\mathbb{R}.$ %
} $\chi_{]-\infty,a]}(\mathrm{m})$ is continuous in $H_{\rho,0}(\mathbb{R};H)$
and by means of Lemma \ref{lem:char_dom} we obtain
\[
\Re\int\limits _{-\infty}^{a}\langle(\partial_{0,\rho}\iota_{1}^{\ast}M_{0}(\mathrm{m})\iota_{1})\iota_{1}^{\ast}u(t)|\iota_{1}^{\ast}u(t)\rangle e^{-2\rho t}\mbox{ d}t\geq(\rho c_{0}-\frac{1}{2}|M_{0}|_{\mathrm{Lip}})\int\limits _{-\infty}^{a}|\iota_{1}^{\ast}u(t)|^{2}e^{-2\rho t}\mbox{ d}t
\]
for all $u\in D(\partial_{0,\rho}M_{0}(\mathrm{m})).$ Moreover, we
compute
\begin{align*}
 & \Re\int\limits _{-\infty}^{a}\langle(\partial_{0,\rho}M_{0}(\mathrm{m})+M_{1}(\mathrm{m}))u(t)|u(t)\rangle e^{-2\rho t}\mbox{ d}t\\
 & =\Re\int\limits _{-\infty}^{a}\langle(\partial_{0,\rho}\iota_{1}^{\ast}M_{0}(\mathrm{m})\iota_{1})\iota_{1}^{\ast}u(t)|\iota_{1}^{\ast}u(t)\rangle e^{-2\rho t}\mbox{ d}t+\Re\int\limits _{-\infty}^{a}\langle\iota_{0}^{\ast}M_{1}(t)\iota_{0}\iota_{0}^{\ast}u(t)|\iota_{0}^{\ast}u(t)\rangle e^{-2\rho t}\mbox{ d}t\\
 & \quad+\Re\int\limits _{-\infty}^{a}\langle\iota_{1}^{\ast}M_{1}(t)\iota_{1}\iota_{1}^{\ast}u(t)|\iota_{1}^{\ast}u(t)\rangle e^{-2\rho t}\mbox{ d}t+\Re\int\limits _{-\infty}^{a}\langle\iota_{1}^{\ast}M_{1}(t)\iota_{0}\iota_{0}^{\ast}u(t)|\iota_{1}^{\ast}u(t)\rangle e^{-2\rho t}\mbox{ d}t\\
 & \quad+\Re\int\limits _{-\infty}^{a}\langle\iota_{0}^{\ast}M_{1}(t)\iota_{1}\iota_{1}^{\ast}u(t)|\iota_{0}^{\ast}u(t)\rangle e^{-2\rho t}\mbox{ d}t\\
 & \geq(\rho c_{0}-\frac{1}{2}|M_{0}|_{\mathrm{Lip}}-|M_{1}|_{\infty})\int\limits _{-\infty}^{a}|\iota_{1}^{\ast}u(t)|^{2}e^{-2\rho t}\mbox{ d}t+c_{1}\int\limits _{-\infty}^{a}|\iota_{0}^{\ast}u(t)|^{2}e^{-2\rho t}\mbox{ d}t\\
 & \quad-2|M_{1}|_{\infty}\int\limits _{-\infty}^{a}|\iota_{0}^{\ast}u(t)||\iota_{1}^{\ast}u(t)|e^{-2\rho t}\mbox{ d}t\\
 & \geq(c_{1}-\epsilon)\int\limits _{-\infty}^{a}|\iota_{0}^{\ast}u(t)|^{2}e^{-2\rho t}\mbox{ d}t+(\rho c_{0}-\frac{1}{2}|M_{0}|_{\mathrm{Lip}}-|M_{1}|_{\infty}-\frac{1}{\epsilon}|M_{1}|_{\infty}^{2})\int\limits _{-\infty}^{a}|\iota_{1}^{\ast}u(t)|^{2}e^{-2\rho t}\mbox{ d}t
\end{align*}
for all $u\in D(\partial_{0,\rho}M_{0}(\mathrm{m})).$ Let now $0<\tilde{c}<c_{1}$
and set $\rho_{0}\coloneqq\frac{1}{c_{0}}\left(\tilde{c}+\frac{1}{2}|M_{0}|_{\mathrm{Lip}}+|M_{1}|_{\infty}+\frac{1}{c_{1}-\tilde{c}}|M_{1}|_{\infty}^{2}\right).$
Then by \prettyref{eq:pos_def_mat_law}, the operator $\partial_{0,\rho}M_{0}(\mathrm{m})+M_{1}(\mathrm{m})-\tilde{c}$
is monotone for each $\rho\geq\rho_{0}.$ To show the maximal monotonicity
of $\partial_{0,\rho}M_{0}(\mathrm{m})+M_{1}(\mathrm{m})-\tilde{c}$
we need to determine the domain of $\left(\partial_{0,\rho}M_{0}(\mathrm{m})\right)^{\ast}.$
Let $v\in D\left(\left(\partial_{0,\rho}M_{0}(\mathrm{m})\right)^{\ast}\right).$
Then for each $\phi\in C_{c}^{\infty}(\mathbb{R};H)$ we compute,
using \prettyref{eq:chainrule}
\begin{align*}
\left\langle \phi\left|\left(\partial_{0,\rho}M_{0}(\mathrm{m})\right)^{\ast}v\right.\right\rangle _{H_{\rho,0}(\mathbb{R};H)} & =\langle\partial_{0,\rho}M_{0}(\mathrm{m})\phi|v\rangle_{H_{\rho,0}(\mathbb{R};H)}\\
 & =\left\langle M_{0}(\mathrm{m})\partial_{0,\rho}\phi|v\right\rangle _{H_{\rho,0}(\mathbb{R};H)}+\langle M_{0}'(\mathrm{m})\phi|v\rangle_{H_{\rho,0}(\mathbb{R};H)}\\
 & =\left\langle \partial_{0,\rho}\phi|M_{0}(\mathrm{m})v\right\rangle _{H_{\rho,0}(\mathbb{R};H)}+\langle\phi|M_{0}'(\mathrm{m})v\rangle_{H_{\rho,0}(\mathbb{R};H)},
\end{align*}
yielding that $M_{0}(\mathrm{m})v\in D(\partial_{0,\rho}^{\ast})=D(\partial_{0,\rho})$,
where we again have used that $C_{c}^{\infty}(\mathbb{R};H)$ is dense
in $H_{\rho,1}(\mathbb{R};H).$ Thus, 
\[
D\left(\left(\partial_{0,\rho}M_{0}(\mathrm{m})\right)^{\ast}\right)\subseteq D\left(\partial_{0,\rho}M_{0}(\mathrm{m})\right),
\]
which yields by \prettyref{eq:pos_def_mat_law} that $\left(\partial_{0,\rho}M_{0}(\mathrm{m})+M_{1}(\mathrm{m})\right)^{\ast}$
is injective and hence, \foreignlanguage{english}{$\partial_{0,\rho}M_{0}(\mathrm{m})+M_{1}(\mathrm{m})$}
is onto. Thus, \prettyref{thm:Minty} implies that $\partial_{0,\rho}M_{0}(\mathrm{m})+M_{1}(\mathrm{m})-\tilde{c}$
is maximal monotone.
\end{proof}
With the latter lemma, the uniqueness of a solution of \prettyref{eq:prob}
and its continuous dependence on the given right hand side for $\rho$
sufficiently large can easily be proved.
\begin{prop}
\label{prop:uniq} Let $0<\tilde{c}<c_{1}$ and $\rho>0$ sufficiently
large such that $\partial_{0,\rho}M_{0}(\mathrm{m})+M_{1}(\mathrm{m})-\tilde{c}$
is maximal monotone. Moreover, let $B\subseteq H_{\rho,0}(\mathbb{R};H)\oplus H_{\rho,0}(\mathbb{R};H)$
be monotone. Then for $(u,f),(v,g)\in\partial_{0,\rho}M_{0}(\mathrm{m})+M_{1}(\mathrm{m})+B$
the estimate 
\[
|u-v|_{H_{\rho,0}(\mathbb{R};H)}\leq\frac{1}{\tilde{c}}|f-g|_{H_{\rho,0}(\mathbb{R};H)}
\]
holds, or, in other words, the inverse relation $\left(\partial_{0,\rho}M_{0}(\mathrm{m})+M_{1}(\mathrm{m})+B\right)^{-1}$
is a Lipschitz-continuous mapping.\end{prop}
\begin{proof}
Let $(u,f),(v,g)\in\partial_{0,\rho}M_{0}(\mathrm{m})+M_{1}(\mathrm{m})+B.$
Then 
\begin{align*}
 & \Re\langle f-g|u-v\rangle_{H_{\rho,0}(\mathbb{R};H)}\\
 & =\Re\left\langle \left.f-\left(\partial_{0,\rho}M_{0}(\mathrm{m})+M_{1}(\mathrm{m})\right)u-\left(g-\left(\partial_{0,\rho}M_{0}(\mathrm{m})+M_{1}(\mathrm{m})\right)v\right)\right|u-v\right\rangle _{H_{\rho,0}(\mathbb{R};H)}\\
 & \quad+\Re\left\langle \left.\left(\partial_{0,\rho}M_{0}(\mathrm{m})+M_{1}(\mathrm{m})\right)(u-v)\right|u-v\right\rangle _{H_{\rho,0}(\mathbb{R};H)}\\
 & \geq\tilde{c}|u-v|_{H_{\rho,0}(\mathbb{R};H)}^{2},
\end{align*}
where we have used the monotonicity of $\partial_{0,\rho}M_{0}(\mathrm{m})+M_{1}(\mathrm{m})-\tilde{c}$
and of $B$. The assertion now follows from the Cauchy-Schwarz-inequality.
\end{proof}
It is left to show that \prettyref{eq:prob} possesses a solution
$u\in H_{\rho,0}(\mathbb{R};H)$ for every right hand side $f\in H_{\rho,0}(\mathbb{R};H).$
Instead of considering the problem \prettyref{eq:prob} we study an
inclusion of the form 
\begin{equation}
(u,f)\in\overline{\partial_{0,\rho}M_{0}(\mathrm{m})-M_{0}'(\mathrm{m})+\delta+A_{\rho}},\label{eq:substitute_prob}
\end{equation}
where $\delta>0.$ This means we specify $M_{1}(m)$ to be of the
form $\delta-M_{0}'(\mathrm{m}).$ It is easy to see that the pair
$(M_{0},\delta-M_{0}')$ satisfies the conditions (a)-(d)%
\footnote{Note that $\iota_{0}^{\ast}M_{0}'(\mathrm{m})\iota_{0}=0$.%
}. We show that for $f\in H_{\rho,1}(\mathbb{R};H)$ there exists $u\in H_{\rho,0}(\mathbb{R};H)$
satisfying \prettyref{eq:substitute_prob}. For doing so, we employ
\prettyref{prop:pert} and thus, we have to consider the approximate
problem 
\begin{equation}
\left(\partial_{0,\rho}M_{0}(\mathrm{m})-M_{0}'(\mathrm{m})+\delta\right)u_{\lambda}+A_{\rho,\lambda}(u_{\lambda})=f,\label{eq:approx_prob}
\end{equation}
for each $\lambda>0$, where we denote by $A_{\rho,\lambda}$ the
Yosida approximation of $A_{\rho}$%
\footnote{One can show that $A_{\rho,\lambda}$ equals the extension of $A_{\lambda}$
given by \prettyref{eq:extension}. %
}. For each $0<\tilde{c}<\delta$ we can choose $\rho_{0}>0$ such
that $\partial_{0,\rho}M_{0}(\mathrm{m})-M_{0}'(\mathrm{m})+\delta-\tilde{c}$
gets maximal monotone for every $\rho\geq\rho_{0}$ by \prettyref{lem:max_mon_material-1}.
Thus, using \prettyref{cor:pert_Lip} we find for each $\lambda>0$
an element $u_{\lambda}\in D(\partial_{0,\rho}M_{0}(\mathrm{m}))$
satisfying \prettyref{eq:approx_prob}. 
\begin{lem}
\label{lem:diff_sol}Let $0<\tilde{c}<\delta$ and $\rho>0$ such
that $\partial_{0,\rho}M_{0}(\mathrm{m})-M_{0}'(\mathrm{m})+\delta-\tilde{c}$
is maximal monotone. Moreover, let $f\in H_{\rho,1}(\mathbb{R};H)$
and $u_{\lambda}\in D(\partial_{0,\rho}M_{0}(\mathrm{m}))$ satisfying
\prettyref{eq:approx_prob} for $\lambda>0$. Then $u_{\lambda}\in H_{\rho,1}(\mathbb{R};H)$.\end{lem}
\begin{proof}
We decompose $u_{\lambda}$ into the orthogonal parts $\iota_{1}^{\ast}u_{\lambda}$
and $\iota_{0}^{\ast}u_{\lambda}$ lying in $H_{\rho,0}(\mathbb{R};N(M_{0})^{\bot})$
and $H_{\rho,0}(\mathbb{R};N(M_{0}))$, respectively. Since $u_{\lambda}\in D(\partial_{0,\rho}M_{0}(\mathrm{m}))$,
\prettyref{lem:char_dom} yields that $\iota_{1}^{\ast}u_{\lambda}\in D(\partial_{0,\rho}).$
Thus, it suffices to prove that also $\iota_{0}^{\ast}u_{\lambda}\in D(\partial_{0,\rho})$,
which will be shown by using \prettyref{prop:chara_diff}. We apply
$\iota_{0}^{\ast}$ on \prettyref{eq:approx_prob} and obtain the
equality 
\begin{equation}
\delta\iota_{0}^{\ast}u_{\lambda}+\iota_{0}^{\ast}A_{\rho,\lambda}(\iota_{0}\iota_{0}^{\ast}u_{\lambda}+\iota_{1}\iota_{1}^{\ast}u_{\lambda})=\iota_{0}^{\ast}f,\label{eq:subspace}
\end{equation}
since $\iota_{0}^{\ast}M_{0}(\mathrm{m})=0=\iota_{0}^{\ast}M_{0}'(\mathrm{m}).$
We define the mapping $B_{\lambda}$ on $H_{\rho,0}(\mathbb{R};N(M_{0}))$
given by 
\[
B_{\lambda}(v)=\iota_{0}^{\ast}A_{\rho,\lambda}(\iota_{0}v+\iota_{1}\iota_{1}^{\ast}u_{\lambda})\quad(v\in H_{\rho,0}(\mathbb{R};N(M_{0}))).
\]
Then \prettyref{eq:subspace} can be written as 
\begin{equation}
\delta\iota_{0}^{\ast}u_{\lambda}+B_{\lambda}(\iota_{0}^{\ast}u_{\lambda})=\iota_{0}^{\ast}f.\label{eq:subspace_2}
\end{equation}
$B_{\lambda}$ is monotone, since for $v,w\in H_{\rho,0}(\mathbb{R};N(M_{0}))$
we estimate 
\begin{align*}
 & \Re\langle B_{\lambda}(v)-B_{\lambda}(w)|v-w\rangle_{H_{\rho,0}(\mathbb{R};N(M_{0}))}\\
 & =\Re\left\langle \left.\iota_{0}^{\ast}\left(A_{\rho,\lambda}(\iota_{0}v+\iota_{1}\iota_{1}^{\ast}u_{\lambda})-A_{\rho,\lambda}(\iota_{0}w+\iota_{1}\iota_{1}^{\ast}u_{\lambda})\right)\right|v-w\right\rangle _{H_{\rho,0}(\mathbb{R};N(M_{0}))}\\
 & =\Re\left\langle A_{\rho,\lambda}(\iota_{0}v+\iota_{1}\iota_{1}^{\ast}u_{\lambda})-A_{\rho,\lambda}(\iota_{0}w+\iota_{1}\iota_{1}^{\ast}u_{\lambda})|\iota_{0}v-\iota_{0}w\right\rangle _{H_{\rho,0}(\mathbb{R};H)}\\
 & =\Re\left\langle A_{\rho,\lambda}(\iota_{0}v+\iota_{1}\iota_{1}^{\ast}u_{\lambda})-A_{\rho,\lambda}(\iota_{0}w+\iota_{1}\iota_{1}^{\ast}u_{\lambda})\left|\left(\iota_{0}v+\iota_{1}\iota_{1}^{\ast}u_{\lambda}\right)-\left(\iota_{0}w+\iota_{1}\iota_{1}^{\ast}u_{\lambda}\right)\right.\right\rangle _{H_{\rho,0}(\mathbb{R};H)}\\
 & \geq0,
\end{align*}
where we have used the monotonicity of $A_{\rho,\lambda}.$ Moreover,
$B_{\lambda}$ is Lipschitz-continuous. Indeed, for $v,w\in H_{\rho,0}(\mathbb{R};N(M_{0}))$
we have that 
\begin{align*}
|B_{\lambda}(v)-B_{\lambda}(w)|_{H_{\rho,0}(\mathbb{R};N(M_{0}))} & =|\iota_{0}^{\ast}A_{\rho,\lambda}(\iota_{0}v+\iota_{1}\iota_{1}^{\ast}u_{\lambda})-\iota_{0}^{\ast}A_{\rho,\lambda}(\iota_{0}w+\iota_{1}\iota_{1}^{\ast}u_{\lambda})|_{H_{\rho,0}(\mathbb{R};N(M_{0}))}\\
 & \leq\lambda^{-1}|\iota_{0}v-\iota_{0}w|_{H_{\rho,0}(\mathbb{R};H)}\\
 & =\lambda^{-1}|v-w|_{H_{\rho,0}(\mathbb{R};N(M_{0}))}.
\end{align*}
Thus, by \prettyref{cor:Lipschitz_maxmon}, $B_{\lambda}$ is maximal
monotone. Hence, we find a unique solution $v\in H_{\rho,0}(\mathbb{R};N(M_{0}))$
of 
\[
\delta v+B_{\lambda}(v)=\iota_{0}^{\ast}f,
\]
which thus coincides with $\iota_{0}^{\ast}u_{\lambda}$ by \prettyref{eq:subspace_2}.
Furthermore, we compute for $h>0$
\begin{align*}
 & \delta\tau_{h}v+B_{\lambda}(\tau_{h}v)\\
 & =\delta\tau_{h}v+\iota_{0}^{\ast}A_{\rho,\lambda}(\iota_{0}\tau_{h}v+\iota_{1}\iota_{1}^{\ast}u_{\lambda})\\
 & =\delta\tau_{h}v+\iota_{0}^{\ast}A_{\rho,\lambda}(\tau_{h}(\iota_{0}v+\iota_{1}\iota_{1}^{\ast}u_{\lambda}))+\iota_{0}^{\ast}A_{\rho,\lambda}(\iota_{0}\tau_{h}v+\iota_{1}\iota_{1}^{\ast}u_{\lambda})-\iota_{0}^{\ast}A_{\rho,\lambda}(\tau_{h}(\iota_{0}v+\iota_{1}\iota_{1}^{\ast}u_{\lambda}))\\
 & =\tau_{h}(\delta v+B_{\lambda}(v))+\iota_{0}^{\ast}A_{\rho,\lambda}(\iota_{0}\tau_{h}v+\iota_{1}\iota_{1}^{\ast}u_{\lambda})-\iota_{0}^{\ast}A_{\rho,\lambda}(\tau_{h}(\iota_{0}v+\iota_{1}\iota_{1}^{\ast}u_{\lambda}))\\
 & =\tau_{h}\iota_{0}^{\ast}f+\iota_{0}^{\ast}A_{\rho,\lambda}(\iota_{0}\tau_{h}v+\iota_{1}\iota_{1}^{\ast}u_{\lambda})-\iota_{0}^{\ast}A_{\rho,\lambda}(\tau_{h}(\iota_{0}v+\iota_{1}\iota_{1}^{\ast}u_{\lambda})),
\end{align*}
where we have used that $A_{\rho,\lambda}\circ\tau_{h}=\tau_{h}\circ A_{\rho,\lambda},$
which follows from \prettyref{eq:extension_autonomous}. Thus, we
estimate 
\begin{align*}
 & \Re\left\langle \tau_{h}\iota_{0}^{\ast}f+\iota_{0}^{\ast}A_{\rho,\lambda}(\iota_{0}\tau_{h}v+\iota_{1}\iota_{1}^{\ast}u_{\lambda})-\iota_{0}^{\ast}A_{\rho,\lambda}(\tau_{h}(\iota_{0}v+\iota_{1}\iota_{1}^{\ast}u_{\lambda}))-\iota_{0}^{\ast}f|\tau_{h}v-v\right\rangle _{H_{\rho,0}(\mathbb{R};N(M_{0}))}\\
 & =\Re\langle\delta\tau_{h}v+B_{\lambda}(\tau_{h}v)-\left(\delta v+B_{\lambda}(v)\right)|\tau_{h}v-v\rangle_{H_{\rho,0}(\mathbb{R};N(M_{0}))}\\
 & \geq\delta|\tau_{h}v-v|_{H_{\rho,0}(\mathbb{R};N(M_{0}))}^{2}
\end{align*}
and hence, by the Cauchy-Schwarz-inequality and the Lipschitz-continuity
of $A_{\rho,\lambda}$
\[
|\tau_{h}v-v|_{H_{\rho,0}(\mathbb{R};N(M_{0}))}\leq\frac{1}{\delta}\left(|\iota_{0}^{\ast}(\tau_{h}f-f)|_{H_{\rho,0}(\mathbb{R};N(M_{0}))}+\frac{1}{\lambda}|\left(\tau_{h}-1\right)\iota_{1}^{\ast}u_{\lambda}|_{H_{\rho,0}(\mathbb{R};N(M_{0})^{\bot})}\right).
\]
Dividing the latter inequality by $h$ and using that $f$ and $\iota_{1}^{\ast}u_{\lambda}$
are in $D(\partial_{0,\rho})$, we derive that $\left(h^{-1}(\tau_{h}-1)v\right)_{h\in]0,1]}$
is bounded, which yields that $\iota_{0}^{\ast}u_{\lambda}=v\in D(\partial_{0,\rho})$
by \prettyref{prop:chara_diff}. This completes the proof.
\end{proof}
We are now able to state the existence result for \prettyref{eq:substitute_prob}.
\begin{prop}
\label{prop:existence}Let $f\in C_{c}^{\infty}(\mathbb{R};H).$ Then
there exists $\rho_{0}>0$ such that for every $\rho\geq\rho_{0}$
we find $u\in H_{\rho,0}(\mathbb{R};H)$ such that 
\[
(u,f)\in\partial_{0,\rho}M_{0}(\mathrm{m})-M_{0}'(\mathrm{m})+\delta+A_{\rho}.
\]
\end{prop}
\begin{proof}
Let $0<\tilde{c}<\delta.$ We choose $\rho_{0}>0$ such that $\partial_{0,\rho}M_{0}(\mathrm{m})-M_{0}'(\mathrm{m})+\delta-\tilde{c}$
and $\partial_{0,\rho}M_{0}(\mathrm{m})+\delta-\tilde{c}$ are maximal
monotone for all $\rho\geq\rho_{0}$%
\footnote{Note that also the pair $(M_{0},\delta)$ satisfies the conditions
(a)-(d) and thus, \prettyref{lem:max_mon_material-1} is applicable.%
}. Let $\rho\geq\rho_{0}.$ For $\lambda>0$ let $u_{\lambda}\in H_{\rho,0}(\mathbb{R};H)$
such that \prettyref{eq:approx_prob} is satisfied. Then, by \prettyref{lem:diff_sol}
$u_{\lambda}\in H_{\rho,1}(\mathbb{R};H).$ In order to show the assertion
we have to prove that the family $\left(A_{\rho,\lambda}(u_{\lambda})\right)_{\lambda>0}$
is bounded (see \prettyref{prop:pert}). For doing so, we define 
\begin{align*}
B_{\lambda}:D(B_{\lambda})\subseteq H_{\rho,0}(\mathbb{R};H) & \to H_{\rho,0}(\mathbb{R};H)\\
v & \mapsto\partial_{0,\rho}A_{\rho,\lambda}(\partial_{0,\rho}^{-1}v)
\end{align*}
for $\lambda>0$ with maximal domain $D(B_{\lambda})\coloneqq\{v\in H_{\rho,0}(\mathbb{R};H)\,|\, A_{\rho,\lambda}(\partial_{0,\rho}^{-1}v)\in H_{\rho,1}(\mathbb{R};H)\}.$
Then for $v\in D(B_{\lambda})$ we estimate, using \prettyref{prop:chara_diff}
\begin{align}
\Re\langle B_{\lambda}(v)|v\rangle_{H_{\rho,0}(\mathbb{R};H)} & =\Re\langle\partial_{0,\rho}A_{\rho,\lambda}(\partial_{0,\rho}^{-1}v)|v\rangle_{H_{\rho,0}(\mathbb{R};H)}\nonumber \\
 & =\lim_{h\to0+}\frac{1}{h^{2}}\Re\langle\tau_{h}A_{\rho,\lambda}(\partial_{0,\rho}^{-1}v)-A_{\rho,\lambda}(\partial_{0,\rho}^{-1}v)|\tau_{h}\partial_{0,\rho}^{-1}v-\partial_{0,\rho}^{-1}v\rangle_{H_{\rho,0}(\mathbb{R};H)}\nonumber \\
 & =\lim_{h\to0+}\frac{1}{h^{2}}\Re\langle A_{\rho,\lambda}(\tau_{h}\partial_{0,\rho}^{-1}v)-A_{\rho,\lambda}(\partial_{0,\rho}^{-1}v)|\tau_{h}\partial_{0,\rho}^{-1}v-\partial_{0,\rho}^{-1}v\rangle_{H_{\rho,0}(\mathbb{R};H)}\nonumber \\
 & \geq0,\label{eq:B_mon}
\end{align}
due to the monotonicity of $A_{\rho,\lambda}.$ Since $u_{\lambda}\in H_{\rho,1}(\mathbb{R};H)$
for each $\lambda>0$ we obtain, using \prettyref{prop:chara_diff}
\begin{align*}
\sup_{h\in]0,1]}\frac{1}{h}\left|\tau_{h}A_{\rho,\lambda}(u_{\lambda})-A_{\rho,\lambda}(u_{\lambda})\right|_{H_{\rho,0}(\mathbb{R};H)} & =\sup_{h\in]0,1]}\frac{1}{h}\left|A_{\rho,\lambda}(\tau_{h}u_{\lambda})-A_{\rho,\lambda}(u_{\lambda})\right|_{H_{\rho,0}(\mathbb{R};H)}\\
 & \leq\frac{1}{\lambda}\sup_{h\in]0,1]}\frac{1}{h}|\tau_{h}u_{\lambda}-u_{\lambda}|_{H_{\rho,0}(\mathbb{R};H)}<\infty
\end{align*}
and thus, again by \prettyref{prop:chara_diff}, $A_{\rho,\lambda}(u_{\lambda})\in H_{\rho,1}(\mathbb{R};H)$
for each $\lambda>0$. In other words, this means that $\partial_{0,\rho}u_{\lambda}\in D(B_{\lambda})$
for each $\lambda>0.$ Moreover, using \prettyref{eq:chainrule} we
obtain 
\begin{align*}
B_{\lambda}(\partial_{0,\rho}u_{\lambda}) & =\partial_{0,\rho}A_{\rho,\lambda}(u_{\lambda})\\
 & =\partial_{0,\rho}\left(f-\left(\partial_{0,\rho}M_{0}(\mathrm{m})-M_{0}'(\mathrm{m})+\delta\right)u_{\lambda}\right)\\
 & =\partial_{0,\rho}f-\partial_{0,\rho}M_{0}(\mathrm{m})\partial_{0,\rho}u_{\lambda}-\delta\partial_{0,\rho}u_{\lambda},
\end{align*}
which means that $\partial_{0,\rho}u_{\lambda}$ solves the differential
equation 
\[
\left(\partial_{0,\rho}M_{0}(\mathrm{m})+\delta+B_{\lambda}\right)(\partial_{0,\rho}u_{\lambda})=\partial_{0,\rho}f
\]
for each $\lambda>0.$ Then we estimate, using \prettyref{eq:B_mon}
and the monotonicity of $\partial_{0,\rho}M_{0}(\mathrm{m})+\delta-\tilde{c}$
\begin{align*}
 & \Re\langle\partial_{0,\rho}f|\partial_{0,\rho}u_{\lambda}\rangle_{H_{\rho,0}(\mathbb{R};H)}\\
 & =\Re\left\langle \left.\left(\partial_{0,\rho}M_{0}(\mathrm{m})+\delta\right)\partial_{0,\rho}u_{\lambda}\right|\partial_{0,\rho}u_{\lambda}\right\rangle _{H_{\rho,0}(\mathbb{R};H)}+\Re\langle B_{\lambda}(\partial_{0,\rho}u_{\lambda})|\partial_{0,\rho}u_{\lambda}\rangle_{H_{\rho,0}(\mathbb{R};H)}\\
 & \geq\tilde{c}|\partial_{0,\rho}u_{\lambda}|_{H_{\rho,0}(\mathbb{R};H)}^{2}
\end{align*}
and hence, 
\[
|\partial_{0,\rho}u_{\lambda}|_{H_{\rho,0}(\mathbb{R};H)}\leq\frac{1}{\tilde{c}}|\partial_{0,\rho}f|_{H_{\rho,0}(\mathbb{R};H)}.
\]
Thus, since $u_{\lambda}$ solves \prettyref{eq:approx_prob}, we
get that 
\begin{align*}
|A_{\rho,\lambda}(u_{\lambda})|_{H_{\rho,0}(\mathbb{R};H)} & =\left|f-\left(\partial_{0,\rho}M_{0}(\mathrm{m})-M_{0}'(\mathrm{m})+\delta\right)u_{\lambda}\right|_{H_{\rho,0}(\mathbb{R};H)}\\
 & \leq|f|_{H_{\rho,0}(\mathbb{R};H)}+|M_{0}(\mathrm{m})\partial_{0,\rho}u_{\lambda}|_{H_{\rho,0}(\mathbb{R};H)}+\delta|u_{\lambda}|_{H_{\rho,0}(\mathbb{R};H)}\\
 & \leq|f|_{H_{\rho,0}(\mathbb{R};H)}+\frac{1}{\tilde{c}}|M_{0}|_{\infty}|\partial_{0,\rho}f|_{H_{\rho,0}(\mathbb{R};H)}+\delta|u_{\lambda}|_{H_{\rho,0}(\mathbb{R};H)}
\end{align*}
for each $\lambda>0.$ Now using that $(0,0)$ also satisfies \prettyref{eq:approx_prob},
we can estimate, using \prettyref{prop:uniq} 
\[
|u_{\lambda}|_{H_{\rho,0}(\mathbb{R};H)}\leq\frac{1}{\tilde{c}}|f|_{H_{\rho,0}(\mathbb{R};H)}
\]
for each $\lambda>0$. Summarizing we get that 
\[
\sup_{\lambda>0}|A_{\rho,\lambda}(u_{\lambda})|_{H_{\rho,0}(\mathbb{R};H)}\leq\left(1+\frac{\delta}{\tilde{c}}\right)|f|_{H_{\rho,0}(\mathbb{R};H)}+\frac{1}{\tilde{c}}|M_{0}|_{\infty}|\partial_{0,\rho}f|_{H_{\rho,0}(\mathbb{R};H)}
\]
for each $\lambda>0$, which completes the proof.
\end{proof}
We summarize our so far findings.
\begin{thm}
\label{thm:sol_th_substitute_prob}For each $\delta>0$ there exists
$\rho_{0}>0$ such that for all $\rho\geq\rho_{0}$ the problem \prettyref{eq:substitute_prob}
is well-posed, i.e. the inverse relation $\left(\overline{\partial_{0,\rho}M_{0}(\mathrm{m})-M_{0}'(\mathrm{m})+\delta+A_{\rho}}\right)^{-1}$
is a Lipschitz-continuous mapping defined on the whole space $H_{\rho,0}(\mathbb{R};H).$
More precisely, for each $0<\tilde{c}<\delta$ there exists $\rho_{0}>0$
such that for each $\rho\geq\rho_{0}$ the relation $\overline{\partial_{0,\rho}M_{0}(\mathrm{m})-M_{0}'(\mathrm{m})+\delta+A_{\rho}}-\tilde{c}\subseteq H_{\rho,0}(\mathbb{R};H)\oplus H_{\rho,0}(\mathbb{R};H)$
is maximal monotone. \end{thm}
\begin{proof}
Let $0<\tilde{c}<\delta.$ By \prettyref{lem:max_mon_material-1}
there exists $\rho_{0}>0$ such that for every $\rho\geq\rho_{0}$
the relation $\partial_{0,\rho}M_{0}(\mathrm{m})-M_{0}'(\mathrm{m})+\delta+A_{\rho}-\tilde{c}$
is monotone, which then also holds for its closure. Moreover, \prettyref{prop:existence}
shows that $\left(\partial_{0,\rho}M_{0}(\mathrm{m})-M_{0}'(\mathrm{m})+\delta+A_{\rho}\right)\left[H_{\rho,0}(\mathbb{R};H)\right]$
contains the test functions (if we choose $\rho$ large enough) and
therefore it is dense in $H_{\rho,0}(\mathbb{R};H)$. The latter yields
that $\left(\overline{\partial_{0,\rho}M_{0}(\mathrm{m})-M_{0}'(\mathrm{m})+\delta+A_{\rho}}\right)\left[H_{\rho,0}(\mathbb{R};H)\right]=H_{\rho,0}(\mathbb{R};H),$
which in turn implies the maximal monotonicity of $\overline{\partial_{0,\rho}M_{0}(\mathrm{m})-M_{0}'(\mathrm{m})+\delta+A_{\rho}}-\tilde{c}$
by \prettyref{thm:Minty}.
\end{proof}
Now we come back to our original problem \prettyref{eq:prob}. It
turns out that the well-posedness for this inclusion just relies on
the perturbation result stated in \prettyref{prop:pert-Lip1}.
\begin{thm}[Well-posedness]
\label{thm:well-posedness} For every $0<\tilde{c}<c_{1}$ there
exists $\rho_{0}>0$ such that for every $\rho\geq\rho_{0}$ the relation
$\overline{\partial_{0,\rho}M_{0}(\mathrm{m})+M_{1}(\mathrm{m})+A_{\rho}}-\tilde{c}$
is maximal monotone. In particular 
\[
\left(\overline{\partial_{0,\rho}M_{0}(\mathrm{m})+M_{1}(\mathrm{m})+A_{\rho}}\right)^{-1}:H_{\rho,0}(\mathbb{R};H)\to H_{\rho,0}(\mathbb{R};H)
\]
is a Lipschitz-continuous mapping. In other words, the problem \prettyref{eq:prob}
is well-posed, i.e. for each right hand side $f\in H_{\rho,0}(\mathbb{R};H)$
there exists a unique $u\in H_{\rho,0}(\mathbb{R};H)$ satisfying
\prettyref{eq:prob} and depending continuously on $f$.\end{thm}
\begin{proof}
Let $0<\tilde{c}<c_{1}.$ According to \prettyref{thm:sol_th_substitute_prob}
there exists $\rho_{0}>0$ such that for all $\rho\geq\rho_{0}$ the
relation 
\[
\overline{\partial_{0,\rho}M_{0}(\mathrm{m})-M_{0}'(\mathrm{m})+2\tilde{c}+A_{\rho}}-\tilde{c}
\]
is maximal monotone. Furthermore, by \prettyref{lem:max_mon_material-1}
there exists $\rho_{1}>0$ such that for all $\rho\geq\rho_{1}$
\[
\overline{\partial_{0,\rho}M_{0}(\mathrm{m})+M_{1}(\mathrm{m})+A_{\rho}}-\tilde{c}
\]
is monotone. Thus, for $\rho\geq\max\{\rho_{0},\rho_{1}\}$ the relation
\[
\overline{\partial_{0,\rho}M_{0}(\mathrm{m})+M_{1}(\mathrm{m})+A_{\rho}}-\tilde{c}=\overline{\partial_{0,\rho}M_{0}(\mathrm{m})-M_{0}'(\mathrm{m})+2\tilde{c}+A_{\rho}}-\tilde{c}+\left(M_{1}(\mathrm{m})+M_{0}'(\mathrm{m})-2\tilde{c}\right)
\]
is maximal monotone by \prettyref{prop:pert-Lip1}.
\end{proof}

\subsection{Causality}

In this section we show that our solution operator $\left(\overline{\partial_{0,\rho}M_{0}(m)+M_{1}(m)+A_{\rho}}\right)^{-1}$
corresponding to the differential inclusion \prettyref{eq:prob} is
causal in $H_{\rho,0}(\mathbb{R};H)$ and independent of the parameter
$\rho$ in the sense that for $f\in H_{\rho,0}(\mathbb{R};H)\cap H_{\nu,0}(\mathbb{R};H)$
we have 
\[
\left(\overline{\partial_{0,\rho}M_{0}(m)+M_{1}(m)+A_{\rho}}\right)^{-1}(f)=\left(\overline{\partial_{0,\nu}M_{0}(m)+M_{1}(m)+A_{\nu}}\right)^{-1}(f)
\]
as functions in $L_{2,\mathrm{loc}}(\mathbb{R};H)$. The definition
of causality in our framework is the following:
\begin{defn}
\label{DefCausal} Let $F:H_{\rho,0}(\mathbb{R};H)\to H_{\rho,0}(\mathbb{R};H)$.
$F$ is called \emph{forward causal} (or simply \emph{causal)}, if
for each $a\in\mathbb{R}$ and $u,v\in H_{\rho,0}(\mathbb{R};H)$
the implication 
\[
\chi_{]-\infty,a]}(\mathrm{m})(u-v)=0\Rightarrow\chi_{]-\infty,a]}(\mathrm{m})\left(F(u)-F(v)\right)=0
\]
holds. Conversely, $F$ is called \emph{backward causal} (or \emph{anticausal),}
if for each $a\in\mathbb{R}$ and $u,v\in H_{\rho,0}(\mathbb{R};H)$
the implication 
\[
\chi_{[a,\infty[}(\mathrm{m})(u-v)=0\Rightarrow\chi_{[a,\infty[}(\mathrm{m})\left(F(u)-F(v)\right)=0
\]
holds.\end{defn}
\begin{prop}
\label{prop:causality}There exists $\rho_{0}>0$ such that for every
$\rho\geq\rho_{0}$ the solution operator $\left(\overline{\partial_{0,\rho}M_{0}(\mathrm{m})+M_{1}(\mathrm{m})+A_{\rho}}\right)^{-1}$
is a causal mapping in $H_{\rho,0}(\mathbb{R};H)$.\end{prop}
\begin{proof}
We choose $\rho_{0}>0$ according to \prettyref{thm:well-posedness}.
Let $a\in\mathbb{R},\,\rho\geq\rho_{0}$ and $f,g\in(\partial_{0,\rho}M_{0}(\mathrm{m})+M_{1}(\mathrm{m})+A_{\rho})[H_{\rho,0}(\mathbb{R};H)]$.
Then there exist two pairs $(u,v),(x,y)\in A_{\rho}$ such that $u,x\in D(\partial_{0,\rho}M_{0}(\mathrm{m}))$
and 
\[
\partial_{0,\rho}M_{0}(\mathrm{m})u+M_{1}(\mathrm{m})u+v=f\quad\text{ and }\quad\partial_{0,\rho}M_{0}(\mathrm{m})x+M_{1}(\mathrm{m})x+y=g.
\]
By using Lemma \ref{lem:max_mon_material-1} and the monotonicity
of $A_{\rho}$ we estimate
\begin{align}
 & \Re\int\limits _{-\infty}^{a}\langle f(t)-g(t)|u(t)-x(t)\rangle e^{-2\rho t}\mbox{ d}t\nonumber \\
 & =\Re\int\limits _{-\infty}^{a}\langle(\partial_{0,\rho}M_{0}(\mathrm{m})+M_{1}(\mathrm{m}))(u-x)(t)+v(t)-y(t)|u(t)-x(t)\rangle e^{-2\rho t}\mbox{ d}t\nonumber \\
 & \geq\tilde{c}\int\limits _{-\infty}^{a}|u(t)-x(t)|^{2}e^{-2\rho t}\mbox{ d}t.\label{eq:causal-1}
\end{align}
Now, let $f,g\in H_{\rho,0}(\mathbb{R};H)$ with $\chi_{]-\infty,a]}(\mathrm{m})(f-g)=0$.
Then there exist two sequences $(f_{n})_{n\in\mathbb{N}},(g_{n})_{n\in\mathbb{N}}$
in $(\partial_{0,\rho}M_{0}(\mathrm{m})+M_{1}(\mathrm{m})+A_{\rho})[H_{\rho,0}(\mathbb{R};H)]$
such that $f_{n}\to f$ and $g_{n}\to g$ in $H_{\rho,0}(\mathbb{R};H)$
as $n\to\infty$. We define
\[
u\coloneqq\left(\overline{\partial_{0,\rho}M_{0}(\mathrm{m})+M_{1}(\mathrm{m})+A_{\rho}}\right)^{-1}f=\lim_{n\to\infty}\left(\partial_{0,\rho}M_{0}(\mathrm{m})+M_{1}(\mathrm{m})+A_{\rho}\right)^{-1}f_{n}
\]
as well as
\[
x\coloneqq\left(\overline{\partial_{0,\rho}M_{0}(\mathrm{m})+M_{1}(\mathrm{m})+A_{\rho}}\right)^{-1}g=\lim_{n\to\infty}\left(\partial_{0,\rho}M_{0}(\mathrm{m})+M_{1}(\mathrm{m})+A_{\rho}\right)^{-1}g_{n}.
\]
Using inequality (\ref{eq:causal-1}) and the continuity of the cut-off
operator $\chi_{]-\infty,a]}(\mathrm{m})$ we obtain
\begin{align*}
0 & =\Re\int\limits _{-\infty}^{a}\langle f(t)-g(t)|u(t)-x(t)\rangle e^{-2\rho t}\mbox{ d}t\\
 & \geq\tilde{c}\int\limits _{-\infty}^{a}|u(t)-x(t)|^{2}e^{-2\rho t}\mbox{ d}t,
\end{align*}
which yields $\chi_{]-\infty,a]}(\mathrm{m})\left(u-x\right)=0.$ 
\end{proof}
The last part of this subsection is devoted to the independence of
the solution operator of problem (\ref{eq:prob}) on the parameter
$\rho>0$. For doing so, we need the following auxiliary results.
\begin{prop}
\label{prop:abstract}Let $X_{0},X_{1}$ be two Banach spaces and
$V$ a vector space with $X_{0},X_{1}\subseteq V$. Moreover, let
$F_{0},G_{0}:X_{0}\to X_{0}$ and $F_{1},G_{1}:X_{1}\to X_{1}$ be
Lipschitz-continuous mappings with $|F_{0}|_{\mathrm{Lip}}\cdot|G_{0}|_{\mathrm{Lip}}<1$
and $|F_{1}|_{\mathrm{Lip}}\cdot|G_{1}|_{\mathrm{Lip}}<1.$ Then for
$i\in\{0,1\}$ the relation $(F_{i}^{-1}+G_{i})^{-1}$ is a Lipschitz-continuous
mapping with domain equal to $X_{i}.$ Furthermore, if $F_{0}|_{X_{0}\cap X_{1}}=F_{1}|_{X_{0}\cap X_{1}}$
and $G_{0}|_{X_{0}\cap X_{1}}=G_{1}|_{X_{0}\cap X_{1}},$ then 
\[
(F_{0}^{-1}+G_{0})^{-1}|_{X_{0}\cap X_{1}}=(F_{1}^{-1}+G_{1})^{-1}|_{X_{0}\cap X_{1}}.
\]
\end{prop}
\begin{proof}
Let $i\in\{0,1\}$. For $x,y\in X_{i}$ we observe that 
\begin{align}
(x,y)\in(F_{i}^{-1}+G_{i})^{-1} & \Leftrightarrow(y,x)\in F_{i}^{-1}+G_{i}\nonumber \\
 & \Leftrightarrow(y,x-G_{i}(y))\in F_{i}^{-1}\nonumber \\
 & \Leftrightarrow y=F_{i}(x-G_{i}(y)).\label{eq:fixed_point}
\end{align}
Since for each $x\in X_{i}$ the mapping $y\mapsto F_{i}(x-G_{i}(y))$
has a unique fixed point, according to the contraction mapping theorem,
we obtain that $(F_{i}^{-1}+G_{i})^{-1}$ is a mapping defined on
the whole space $X_{i}.$ Moreover, using \prettyref{eq:fixed_point},
we estimate for $(x_{0},y_{0}),(x_{1},y_{1})\in(F_{i}^{-1}+G_{i})^{-1}$
\begin{align*}
|y_{0}-y_{1}|_{X_{i}} & =|F_{i}(x_{0}-G_{i}(y_{0}))-F_{i}(x_{1}-G_{i}(y_{1}))|_{X_{i}}\\
 & \leq|F_{i}|_{\mathrm{Lip}}\left(|x_{0}-x_{1}|_{X_{i}}+|G_{i}(y_{0})-G_{i}(y_{1})|_{X_{i}}\right)\\
 & \leq|F_{i}|_{\mathrm{Lip}}|x_{0}-x_{1}|_{X_{i}}+|F_{i}|_{\mathrm{Lip}}|G_{i}|_{\mathrm{Lip}}|y_{0}-y_{1}|_{X_{i}},
\end{align*}
and thus 
\[
|y_{0}-y_{1}|_{X_{i}}\leq\frac{|F_{i}|_{\mathrm{Lip}}}{1-|F_{i}|_{\mathrm{Lip}}|G_{i}|_{\mathrm{Lip}}}|x_{0}-x_{1}|_{X_{i}},
\]
which proves the Lipschitz-continuity of $(F_{i}^{-1}+G_{i})^{-1}.$
Now assume that $F_{0}|_{X_{0}\cap X_{1}}=F_{1}|_{X_{0}\cap X_{1}}$
and $G_{0}|_{X_{0}\cap X_{1}}=G_{1}|_{X_{0}\cap X_{1}}$ and let $x\in X_{0}\cap X_{1}.$
Set $y_{0}\coloneqq0\in X_{0}\cap X_{1}$. For $n\in\mathbb{N}$ we
define 
\[
y_{n+1}\coloneqq F_{0}(x-G_{0}(y_{n})).
\]
Noting that for $z\in X_{0}\cap X_{1}$ we have $G_{0}(z)=G_{1}(z)\in X_{0}\cap X_{1}$
and $F_{0}(x-G_{0}(z))=F_{1}(x-G_{1}(z))\in X_{0}\cap X_{1},$ we
can show inductively that $y_{n}\in X_{0}\cap X_{1}$ for every $n\in\mathbb{N}$.
Moreover, by the contraction mapping theorem and \prettyref{eq:fixed_point}
we get that 
\begin{align*}
(F_{0}^{-1}+G_{0})^{-1}(x) & =\lim_{n\to\infty}y_{n+1}\\
 & =\lim_{n\to\infty}F_{0}(x-G_{0}(y_{n}))\\
 & =\lim_{n\to\infty}F_{1}(x-G_{1}(y_{n}))\\
 & =(F_{1}^{-1}+G_{1})^{-1}(x).\tag*{\qedhere}
\end{align*}
\end{proof}
\begin{lem}
\label{lem:embeddings} Let $\nu\geq\rho>0,$ $a\in\mathbb{R}$ and
$u\in H_{\rho,0}(\mathbb{R};H)$ such that the support $\spt u\subseteq[a,\infty[$.
Then $u\in H_{\nu,0}(\mathbb{R};H).$ Moreover, if $u\in H_{\rho,1}(\mathbb{R};H)$
then $u\in H_{\nu,1}(\mathbb{R};H)$ with 
\[
\partial_{0,\rho}u=\partial_{0,\nu}u.
\]
\end{lem}
\begin{proof}
We estimate 
\begin{equation}
\intop_{\mathbb{R}}\left|u(t)\right|^{2}e^{-2\nu t}\mbox{ d}t=\intop_{a}^{\infty}|u(t)|^{2}e^{-2\nu t}\mbox{ d}t\leq|u|_{H_{\rho,0}(\mathbb{R};H)}^{2}e^{2(\rho-\nu)a},\label{eq:rho_nu}
\end{equation}
which yields the first assertion. Assume now that $u\in H_{\rho,1}(\mathbb{R};H).$
Let $\phi\in C_{c}^{\infty}(\mathbb{R};H)$ and define $\psi(t)\coloneqq e^{2(\rho-\nu)t}\phi(t)$
for $t\in\mathbb{R}.$ Then clearly $\psi\in C_{c}^{\infty}(\mathbb{R};H)$
and we compute
\begin{align*}
\langle u|\partial_{0,\nu}^{\ast}\phi\rangle_{H_{\nu,0}(\mathbb{R};H)} & =\intop_{a}^{\infty}\langle u(t)|-\phi'(t)+2\nu\phi(t)\rangle e^{-2\nu t}\mbox{ d}t\\
 & =\intop_{a}^{\infty}\langle u(t)|-\psi'(t)+2\rho\psi(t)\rangle e^{-2\rho t}\mbox{ d}t\\
 & =\langle u|\partial_{0,\rho}^{\ast}\psi\rangle_{H_{\rho,0}(\mathbb{R};H)}\\
 & =\langle\partial_{0,\rho}u|\psi\rangle_{H_{\rho,0}(\mathbb{R};H)}\\
 & =\intop_{\mathbb{R}}\left\langle \left.\left(\partial_{0,\rho}u\right)(t)\right|\phi(t)\right\rangle e^{-2\nu t}\mbox{ d}t.
\end{align*}
If $\phi$ is supported on $]-\infty,a]$, the latter yields that
\[
\intop_{\mathbb{R}}\left\langle \left.\left(\partial_{0,\rho}u\right)(t)\right|\phi(t)\right\rangle e^{-2\nu t}\mbox{ d}t=0.
\]
Thus, the fundamental lemma of variational calculus implies $\spt\partial_{0,\rho}u\subseteq[a,\infty[$
and thus, $\partial_{0,\rho}u\in H_{\nu,0}(\mathbb{R};H)$ by \prettyref{eq:rho_nu}.
Hence, by the computation above 
\[
\langle u|\partial_{0,\nu}^{\ast}\phi\rangle_{H_{\nu,0}(\mathbb{R};H)}=\langle\partial_{0,\rho}u|\phi\rangle_{H_{\nu,0}(\mathbb{R};H)}
\]
for each $\phi\in C_{c}^{\infty}(\mathbb{R};H)$, which yields the
assertion, since $C_{c}^{\infty}(\mathbb{R};H)$ is a core for $\partial_{0,\nu}^{\ast}.$ 
\end{proof}
We are now able to prove the independence of the parameter $\rho$
of the solution operator associated with \prettyref{eq:substitute_prob}. 
\begin{lem}
\label{lem:IndependenceParameter}Let $\delta>0$ and choose $\rho_{0}>0$
according to \prettyref{thm:sol_th_substitute_prob}. Let $\nu\geq\rho\geq\rho_{0}$
and $f\in H_{\rho,0}(\mathbb{R};H)\cap H_{\nu,0}(\mathbb{R};H)$.
Then
\[
\left(\overline{\partial_{0,\rho}M_{0}(\mathrm{m})-M_{0}'(\mathrm{m})+\delta+A_{\rho}}\right)^{-1}(f)=\left(\overline{\partial_{0,\nu}M_{0}(\mathrm{m})-M_{0}'(\mathrm{m})+\delta+A_{\nu}}\right)^{-1}(f).
\]
\end{lem}
\begin{proof}
We begin to show the assertion for $f\in C_{c}^{\infty}(\mathbb{R};H)\subseteq H_{\rho,0}(\mathbb{R};H)\cap H_{\nu,0}(\mathbb{R};H)$.
Then there exists $a\in\mathbb{R}$ such that $\chi_{[a,\infty[}(m)f=f.$
Due to Proposition \ref{prop:existence} there exist $u_{\nu}\in D(\partial_{0,\nu}M_{0}(\mathrm{m}))$
and $u_{\rho}\in D(\partial_{0,\rho}M_{0}(\mathrm{m}))$ such that
\[
(u_{\rho},f)\in\partial_{0,\rho}M_{0}(\mathrm{m})-M_{0}'(\mathrm{m})+\delta+A_{\rho}\mbox{ and }(u_{\nu},f)\in\partial_{0,\nu}M_{0}(\mathrm{m})-M_{0}'(\mathrm{m})+\delta+A_{\nu}.
\]
Furthermore by using \prettyref{prop:causality} and $(0,0)\in A$
we get that $\spt u_{\rho}\subseteq[a,\infty[$ and hence $u_{\rho}\in H_{\nu,0}(\mathbb{R};H)$
and 
\[
\partial_{\text{0,\ensuremath{\rho}}}M_{0}(\mathrm{m})u_{\rho}=\partial_{0,\nu}M_{0}(\mathrm{m})u_{\rho}
\]
by \prettyref{lem:embeddings}. Thus, we get that 
\[
(u_{\rho},f)\in\partial_{0,\nu}M_{0}(\mathrm{m})-M_{0}'(\mathrm{m})+\delta+A_{\nu}
\]
and by the uniqueness of the solution, this yields $u_{\rho}=u_{\nu}.$
Now let $f\in H_{\rho,0}(\mathbb{R};H)\cap H_{\nu,0}(\mathbb{R};H)$.
Then there exists a sequence $(f_{n})_{n\in\mathbb{N}}\in C_{c}^{\infty}(\mathbb{R};H)^{\mathbb{N}}$
which converges to $f$ in both spaces $H_{\rho,0}(\mathbb{R};H)$
and $H_{\nu,0}(\mathbb{R};H)$ (cf. \cite[Lemma 3.5]{Trostorff2013_stability}).
Due to the continuity of the solution operators and what we have shown
above, we get that 
\begin{align*}
\left(\overline{\partial_{0,\nu}M_{0}(\mathrm{m})-M_{0}'(\mathrm{m})+\delta+A_{\nu}}\right)^{-1}(f) & =\lim_{n\to\infty}\left(\partial_{0,\nu}M_{0}(\mathrm{m})-M_{0}'(\mathrm{m})+\delta+A_{\nu}\right)^{-1}(f_{n})\\
 & =\lim_{n\to\infty}\left(\partial_{0,\rho}M_{0}(\mathrm{m})-M_{0}'(\mathrm{m})+\delta+A_{\rho}\right)^{-1}(f_{n})\\
 & =\left(\overline{\partial_{0,\rho}M_{0}(\mathrm{m})-M_{0}'(\mathrm{m})+\delta+A_{\rho}}\right)^{-1}(f).\tag*{\qedhere}
\end{align*}

\end{proof}
Finally, we prove the independence of the parameter $\rho$ of the
solution operator corresponding to our original problem \prettyref{eq:prob}.
\begin{prop}
There exists $\rho_{0}>0$ such that for every $\nu\geq\rho\geq\rho_{0}$
and $f\in H_{\rho,0}(\mathbb{R};H)\cap H_{\nu,0}(\mathbb{R};H$) we
have that 
\[
\left(\overline{\partial_{0,\rho}M_{0}(\mathrm{m})+M_{1}(\mathrm{m})+A_{\rho}}\right)^{-1}(f)=\left(\overline{\partial_{0,\nu}M_{0}(\mathrm{m})+M_{1}(\mathrm{m})+A_{\nu}}\right)^{-1}(f).
\]
\end{prop}
\begin{proof}
Let $\delta>2(|M_{1}|_{\infty}+|M_{0}|_{\mathrm{Lip}})$. The proof
will be done in two steps. 

\begin{enumerate}[(a)] 

\item According to \prettyref{thm:sol_th_substitute_prob} there
exists $\rho_{1}>0$ such that for every $\rho\geq\rho_{1}$ the inverse
relation $\left(\overline{\partial_{0,\rho}M_{0}(\mathrm{m})-M_{0}'(\mathrm{m})+\delta+A_{\rho}}\right)^{-1}$
is a Lipschitz-continuous mapping with 
\[
\left|\left(\overline{\partial_{0,\rho}M_{0}(\mathrm{m})-M_{0}'(\mathrm{m})+\delta+A_{\rho}}\right)^{-1}\right|_{\mathrm{Lip}}\leq\frac{2}{\delta}.
\]
For $\nu\geq\rho\geq\rho_{1}$ we set $X_{0}\coloneqq H_{\rho,0}(\mathbb{R};H)$
and $X_{1}\coloneqq H_{\nu,0}(\mathbb{R};H)$. Moreover, we define
$F_{0}\coloneqq\left(\overline{\partial_{0,\rho}M_{0}(\mathrm{m})-M_{0}'(\mathrm{m})+\delta+A_{\rho}}\right)^{-1}$
and $F_{1}\coloneqq\left(\overline{\partial_{0,\nu}M_{0}(\mathrm{m})-M_{0}'(\mathrm{m})+\delta+A_{\nu}}\right)^{-1}$.
Then by \prettyref{lem:IndependenceParameter} we have that 
\[
F_{0}|_{X_{0}\cap X_{1}}=F_{1}|_{X_{0}\cap X_{1}}.
\]
Moreover, we set $G_{0}$ and $G_{1}$ as $M_{1}(\mathrm{m})+M_{0}'(\mathrm{m})$
interpreted as bounded linear operators in $X_{0}$ and $X_{1}$,
respectively. Then by definition 
\[
G_{0}|_{X_{0}\cap X_{1}}=G_{1}|_{X_{0}\cap X_{1}.}
\]
Furthermore, $|F_{i}|_{\mathrm{Lip}}\cdot|G_{i}|_{\mathrm{Lip}}\leq\frac{1}{2}$
for $i\in\{0,1\}$ and thus, by \prettyref{prop:abstract}
\begin{align*}
\left(\overline{\partial_{0,\rho}M_{0}(\mathrm{m})+M_{1}(\mathrm{m})+\delta+A_{\rho}}\right)^{-1}|_{X_{0}\cap X_{1}} & =\left(F_{0}^{-1}+G_{0}\right)^{-1}|_{X_{0}\cap X_{1}}\\
 & =(F_{1}^{-1}+G_{1})^{-1}|_{X_{0}\cap X_{1}}\\
 & =\left(\overline{\partial_{0,\nu}M_{0}(\mathrm{m})+M_{1}(\mathrm{m})+\delta+A_{\nu}}\right)^{-1}|_{X_{0}\cap X_{1}}.
\end{align*}

\item Let $\rho_{1}$ be as in (a). By \prettyref{thm:well-posedness}
there exists $\rho_{0}\geq\rho_{1}$ such that for every $\rho\geq\rho_{0}$
the inverse relation \foreignlanguage{english}{$\left(\overline{\partial_{0,\rho}M_{0}(\mathrm{m})+M_{1}(\mathrm{m})+\delta+A_{\rho}}\right)^{-1}$}is
a Lipschitz-continuous mapping with 
\[
\left|\left(\overline{\partial_{0,\rho}M_{0}(\mathrm{m})+M_{1}(\mathrm{m})+\delta+A_{\rho}}\right)^{-1}\right|_{\mathrm{Lip}}<\frac{1}{\delta}.
\]
Let $\nu\geq\rho\geq\rho_{0}$ and set $X_{0}\coloneqq H_{\rho,0}(\mathbb{R};H)$
and $X_{1}\coloneqq H_{\nu,0}(\mathbb{R};H)$. Moreover, we define
$F_{0}\coloneqq\left(\overline{\partial_{0,\rho}M_{0}(\mathrm{m})+M_{1}(\mathrm{m})+\delta+A_{\rho}}\right)^{-1}$
and $F_{1}\coloneqq\left(\overline{\partial_{0,\nu}M_{0}(\mathrm{m})+M_{1}(\mathrm{m})+\delta+A_{\nu}}\right)^{-1}$
and by (a) we have 
\[
F_{0}|_{X_{0}\cap X_{1}}=F_{1}|_{X_{0}\cap X_{1}}.
\]
Defining $G_{0}$ and $G_{1}$ as $-\delta$, interpreted as bounded
linear operators in $X_{0}$ and $X_{1},$ respectively, we derive
the assertion by using \prettyref{prop:abstract}.\qedhere 

\end{enumerate}
\end{proof}

\section{Examples}

In this section we apply our solution theory to two concrete systems
out of the theory of plasticity. The first one, dealing with thermoplasticity,
couples the heat equation with the equations of plasticity, where
the stress tensor and the inelastic part of the strain tensor are
linked via a maximal monotone relation. The second example deals with
the equations of viscoplasticity, where the inelastic strain tensor
is given in terms of an internal variable, satisfying a differential
inclusion. Before we can state the equations, we have to define the
differential operators involved. Throughout, let $\Omega\subseteq\mathbb{R}^{3}$
be an arbitrary domain.
\begin{defn}
We define the operator $\grad_{c}$ (the gradient with Dirichlet-type
boundary conditions) as the closure of 
\begin{align*}
\grad|_{C_{c}^{\infty}(\Omega)}:C_{c}^{\infty}(\Omega)\subseteq L_{2}(\Omega) & \to L_{2}(\Omega)^{3}\\
\phi & \mapsto\left(\partial_{i}\phi\right)_{i\in\{1,2,3\}}
\end{align*}
and the operator $\dive_{c}$ (the divergence with Neumann-type boundary
conditions) as the closure of 
\begin{align*}
\dive|_{C_{c}^{\infty}(\Omega)^{3}}:C_{c}^{\infty}(\Omega)^{3}\subseteq L_{2}(\Omega)^{3} & \to L_{2}(\Omega)\\
(\psi_{i})_{i\in\{1,2,3\}} & \mapsto\sum_{i=1}^{3}\partial_{i}\psi_{i}.
\end{align*}
Using integration by parts one easily sees that 
\begin{align*}
\grad_{c} & \subseteq\left(-\dive_{c}\right)^{\ast},\\
\dive_{c} & \subseteq\left(-\grad_{c}\right)^{\ast}.
\end{align*}
We set $\grad\coloneqq\left(-\dive_{c}\right)^{\ast}$ (the distibutional
gradient with maximal domain in $L_{2}(\Omega)$) and $\dive\coloneqq\left(-\grad_{c}\right)^{\ast}$
(the distributional divergence with maximal domain in $L_{2}(\Omega)^{3}$).
Moreover, we define the matrix-valued symmetrized gradient and the
vector-valued divergence, by setting $\Grad_{c}$ to be the closure
of 
\begin{align*}
\Grad|_{C_{c}^{\infty}(\Omega)^{3}}:C_{c}^{\infty}(\Omega)^{3}\subseteq L_{2}(\Omega)^{3} & \to L_{2,\mathrm{sym}}(\Omega)^{3\times3}\\
\left(\psi_{i}\right)_{i\in\{1,2,3\}} & \mapsto\frac{1}{2}\left(\partial_{i}\psi_{j}+\partial_{j}\psi_{i}\right)_{i,j\in\{1,2,3\}}
\end{align*}
and by defining $\Dive_{c}$ as the closure of 
\begin{align*}
\Dive|_{C_{c,\mathrm{sym}}^{\infty}(\Omega)^{3\times3}}:C_{c,\mathrm{sym}}^{\infty}(\Omega)^{3\times3}\subseteq L_{2,\mathrm{sym}}(\Omega)^{3\times3} & \to L_{2}(\Omega)^{3}\\
\left(T_{ij}\right)_{i,j\in\{1,2,3\}} & \mapsto\left(\sum_{j=1}^{3}\partial_{j}T_{ij}\right)_{i\in\{1,2,3\}}.
\end{align*}
Here $L_{2,\mathrm{sym}}(\Omega)^{3\times3}$ denotes the space of
$L_{2}$ functions attaining values in the space of symmetric $3\times3$
matrices, equipped with the Frobenius inner product 
\[
\langle T|S\rangle_{L_{2,\mathrm{sym}}(\Omega)^{3\times3}}\coloneqq\intop_{\Omega}\trace\left(T(x)^{\ast}S(x)\right)\mbox{ d}x\quad(T,S\in L_{2,\mathrm{sym}}(\Omega)^{3\times3})
\]
and $C_{c,\mathrm{sym}}^{\infty}(\Omega)^{3\times3}$ denotes the
space of test functions with values in the space of symmetric $3\times3$
matrices. Like in the scalar-valued case we obtain 
\begin{align*}
\Grad_{c} & \subseteq\left(-\Dive_{c}\right)^{\ast}\eqqcolon\Grad,\\
\Dive_{c} & \subseteq\left(-\Grad_{c}\right)^{\ast}\eqqcolon\Dive.
\end{align*}

\end{defn}

\subsection{Thermoplasticity}

We denote by $u:\mathbb{R}\to L_{2}(\Omega)^{3}$ the displacement
field, by $\theta:\mathbb{R}\to L_{2}(\Omega)$ the temperature density
and by $T:\mathbb{R}\to L_{2,\mathrm{sym}}(\Omega)^{3\times3}$ the
stress tensor of a body $\Omega\subseteq\mathbb{R}^{3}.$ The equations
of thermoplasticity read as 
\begin{align}
\partial_{0,\rho}M\partial_{0,\rho}u-\Dive T & =f,\label{eq:elast}\\
\partial_{0,\rho}w\theta-\dive\kappa\grad\theta+\tau_{0}\trace\Grad\partial_{0,\rho}u & =g.\label{eq:heat}
\end{align}
Here $w:\mathbb{R}\to L_{\infty}(\Omega),$ modelling the time-dependent
mass density, is assumed to be bounded and Lipschitz-continuous, $w(t)$
is real-valued for each $t\in\mathbb{R}$ and $w$ is uniformly positive%
\footnote{This means that there exists a positive constant $c>0$ such that
$w(t)\geq c$ for each $t\in\mathbb{R}$. %
}. Likewise $M:\mathbb{R}\to L(L_{2}(\Omega)^{3})$ is a bounded, Lipschitz-continuous
function, $M(t)$ is selfadjoint for each $t\in\mathbb{R}$ and $M$
is uniformly strictly positive definite and $\tau_{0}>0$ is a numerical
parameter. The function $\kappa:\mathbb{R}\to L(L_{2}(\Omega)^{3})$
describes the (time-dependent) heat conductivity of $\Omega$ and
is assumed to be measurable, bounded and $\Re\kappa(t)=\frac{1}{2}\left(\kappa(t)+\kappa(t)^{\ast}\right)\geq\tilde{c}>0$
for every $t\in\mathbb{R}.$ The right hand sides $f:\mathbb{R}\to L_{2}(\Omega)^{3}$
and $g:\mathbb{R}\to L_{2}(\Omega)$ are given source terms. The operator
$\trace$ is defined by 
\begin{align*}
\trace:L_{2,\mathrm{sym}}(\Omega)^{3\times3} & \to L_{2}(\Omega)\\
\left(S_{ij}\right)_{i,j\in\{1,2,3\}} & \mapsto\sum_{i=1}^{3}S_{ii}.
\end{align*}
The stress tensor $T$ is coupled with the strain tensor $\Grad u$
via the constitutive relation%
\footnote{The adjoint of the operator $\trace$ is given by 
\[
\trace^{\ast}f=\left(\begin{array}{ccc}
f & 0 & 0\\
0 & f & 0\\
0 & 0 & f
\end{array}\right)\quad(f\in L_{2}(\Omega)).
\]
} 
\begin{equation}
T=C\left(\Grad u-\varepsilon_{p}\right)-c\trace^{\ast}\theta,\label{eq:Hooke}
\end{equation}
where $C:\mathbb{R}\to L\left(L_{2,\mathrm{sym}}(\Omega)^{3\times3}\right)$
is bounded, Lipschitz-continuous and uniformly strictly positive definite,
$C(t)$ is selfadjoint for every $t\in\mathbb{R}$ and $c>0$ is a
coupling parameter. The function $\varepsilon_{p}:\mathbb{R}\to L_{2,\mathrm{sym}}(\Omega)^{3\times3}$
denotes the inelastic part of the strain tensor and is linked to $T$
by 
\begin{equation}
(T,\partial_{0,\rho}\varepsilon_{p})\in\mathbb{I},\label{eq:plastic}
\end{equation}
where $\mathbb{I}\subseteq L_{2,\mathrm{sym}}(\Omega)^{3\times3}\oplus L_{2,\mathrm{sym}}(\Omega)^{3\times3}$
is a bounded maximal monotone relation with $(0,0)\in\mathbb{I}$
and for every element $S\in\mathbb{I}\left[L_{2,\mathrm{sym}}(\Omega)^{3\times3}\right]$
we have that $\trace S=0.$ This system was considered for the autonomous
case in \cite{Trostorff2012_NA} and earlier by \cite{Chelminski2006}
for the autonomous, quasi-static case. Following \cite{Trostorff2012_NA},
we rewrite the system in the following way: Define $v\coloneqq\partial_{0,\rho}u$
and $q\coloneqq c\tau_{0}^{-1}\kappa\grad\theta.$ Then \prettyref{eq:elast}
and \prettyref{eq:heat} can be written as 
\begin{align}
\partial_{0,\rho}Mv-\Dive T & =f,\nonumber \\
\partial_{0,\rho}c\tau_{0}^{-1}w\theta-\dive q+\trace c\Grad v & =c\tau_{0}^{-1}g.\label{eq:heat2}
\end{align}
Moreover, applying $\partial_{0,\rho}C^{-1}$ to \prettyref{eq:Hooke}
yields%
\footnote{Note that $C(t)$ is invertible for each $t\in\mathbb{R}$ and $C^{-1}:t\mapsto C(t)^{-1}$
is bounded, Lipschitz-continuous and uniformly strictly positive definite.%
} 
\begin{equation}
\partial_{0,\rho}C^{-1}T+\partial_{0,\rho}\varepsilon_{p}+\partial_{0,\rho}C^{-1}c\trace^{\ast}\theta=\Grad v,\label{eq:Hooke2}
\end{equation}
which gives 
\[
\partial_{0,\rho}\trace cC^{-1}T+\partial_{0,\rho}\trace cC^{-1}c\trace^{\ast}\theta=\trace c\Grad v,
\]
where we have used that $\partial_{0,\rho}\varepsilon_{p}\in\mathbb{I}\left[L_{2,\mathrm{sym}}(\Omega)^{3\times3}\right]\subseteq[\{0\}]\trace$
by \prettyref{eq:plastic}. Using this representation of $\trace c\Grad v,$
\prettyref{eq:heat2} reads as 
\[
\partial_{0,\rho}\left(c\tau_{0}^{-1}w+\trace cC^{-1}c\trace^{\ast}\right)\theta+\partial_{0,\rho}\trace cC^{-1}T-\dive q=c\tau_{0}^{-1}g.
\]
Using these equations, the system \prettyref{eq:elast}-\prettyref{eq:plastic}
can be written as 
\begin{multline*}
\left(\left(\begin{array}{c}
v\\
T\\
\theta\\
q
\end{array}\right),\left(\begin{array}{c}
f\\
0\\
g\\
0
\end{array}\right)\right)\in\partial_{0,\rho}\left(\begin{array}{cccc}
M & 0 & 0 & 0\\
0 & C^{-1} & C^{-1}c\trace^{\ast} & 0\\
0 & \trace cC^{-1} & c\tau_{0}^{-1}w+\trace cC^{-1}c\trace^{\ast} & 0\\
0 & 0 & 0 & 0
\end{array}\right)+\left(\begin{array}{cccc}
0 & 0 & 0 & 0\\
0 & 0 & 0 & 0\\
0 & 0 & 0 & 0\\
0 & 0 & 0 & \kappa^{-1}c^{-1}\tau_{0}
\end{array}\right)\\
+\left(\begin{array}{cccc}
0 & -\Dive & 0 & 0\\
-\Grad & \mathbb{I} & 0 & 0\\
0 & 0 & 0 & -\dive\\
0 & 0 & -\grad & 0
\end{array}\right).
\end{multline*}
Indeed, this system fits into our general framework with 
\[
M_{0}(t)\coloneqq\left(\begin{array}{cccc}
M(t) & 0 & 0 & 0\\
0 & C(t)^{-1} & C(t)^{-1}c\trace^{\ast} & 0\\
0 & \trace cC(t)^{-1} & c\tau_{0}^{-1}w(t)+\trace cC(t)^{-1}c\trace^{\ast} & 0\\
0 & 0 & 0 & 0
\end{array}\right)
\]
and 
\[
M_{1}(t)\coloneqq\left(\begin{array}{cccc}
0 & 0 & 0 & 0\\
0 & 0 & 0 & 0\\
0 & 0 & 0 & 0\\
0 & 0 & 0 & \kappa(t)^{-1}c^{-1}\tau_{0}
\end{array}\right)
\]
and 
\[
A\coloneqq\left(\begin{array}{cccc}
0 & -\Dive & 0 & 0\\
-\Grad & \mathbb{I} & 0 & 0\\
0 & 0 & 0 & -\dive\\
0 & 0 & -\grad & 0
\end{array}\right).
\]
The assumptions on the coefficient ensure that our solvability conditions
(a)-(d) are satisfied. Moreover, by imposing suitable boundary conditions,
e.g. Dirichlet boundary conditions for $v$ and $\theta$, we obtain
that 
\[
\left(\begin{array}{cccc}
0 & -\Dive & 0 & 0\\
-\Grad_{c} & 0 & 0 & 0\\
0 & 0 & 0 & -\dive\\
0 & 0 & -\grad_{c} & 0
\end{array}\right)
\]
is skew-selfadjoint and hence, maximal monotone. Since $\mathbb{I}$
is bounded and maximal monotone, \prettyref{cor:bounded_pert} yields
the maximal monotonicity of $A$. Moreover, $(0,0)\in A$ by assumption.

\subsection{Viscoplasticity with internal variables}

In this section we consider the equations of viscoplasticity with
internal variables. The problem belongs to the class of constitutive
equations which are studied in \cite{Alber_1998}.

As in the thermoplastic case we denote by $u\colon\R\to L_{2}(\Omega)^{3}$
the displacement field of a medium $\Omega\subseteq\mathbb{R}^{3}$
and by $T\colon\R\to L_{2,\mathrm{sym}}(\Omega)^{3\times3}$ the stress
tensor. Moreover, $z\colon\R\to L_{2}(\Omega)^{N}$, where $N\in\mathbb{N}$,
is the vector of internal variables. The model equations of viscoplasticity
with internal variables are 
\begin{align}
\partial_{0,\rho}M\partial_{0,\rho}u-\Dive T & =f,\label{eq:visco1}\\
T & =D(\Grad u-Bz),\label{eq:visco2}\\
(B^{*}T-Lz,\partial_{0,\rho}z) & \in g,\label{eq:visco3}
\end{align}

where $M\colon\R\to L(L_{2}(\Omega)^{3})$, the elasticity tensor
$D\colon\R\to L(L_{2,\mathrm{sym}}(\Omega)^{3\times3})$ and $L\colon\R\to L(L_{2}(\Omega)^{N})$
are bounded, Lipschitz-continuous, uniformly strictly positive definite
functions and $M(t),D(t),L(t)$ are selfadjoint for each $t\in\R$.
Furthermore, $g\subseteq L_{2}(\Omega)^{N}\oplus L_{2}(\Omega)^{N}$
is a bounded, maximal monotone relation with $(0,0)\in g$ and $B\colon L_{2}(\Omega)^{N}\to L_{2,\mathrm{sym}}(\Omega)^{3\times3}$
is linear and continuous. The mapping $B$ is linked to the inelastic
part of the strain tensor $\epsilon=\Grad u$ by $\epsilon_{p}=Bz$.
The volume force $f\colon\R\to L_{2}(\Omega)^{3}$ is given. This
system was studied in \cite{Alber2009} in the autonomous, quasi-static
case, where the focus was on the regularity of solutions.

To apply our solution theory to these equations we have to rewrite
the system. For doing so, we define $v:=\partial_{0,\rho}u$ and $w:=B^{*}T-Lz$.
Thus, we obtain $z=L^{-1}(B^{*}T-w)$ and we can reformulate (\ref{eq:visco1})
and (\ref{eq:visco3}) by
\begin{align*}
\partial_{0,\rho}Mv-\Dive T & =f,\\
(w,\partial_{0,\rho}L^{-1}(B^{*}T-w)) & \in g.
\end{align*}

Moreover, applying $\partial_{0,\rho}D^{-1}$ to equation (\ref{eq:visco2})
yields
\[
\partial_{0,\rho}D^{-1}T=\partial_{0,\rho}(\Grad u-BL^{-1}(B^{*}T-w))=\Grad v-\partial_{0,\rho}BL^{-1}B^{*}T+\partial_{0,\rho}BL^{-1}w.
\]

Hence, the system (\ref{eq:visco1})-(\ref{eq:visco3}) can be written
as
\[
\left(\left(\begin{array}{c}
v\\
w\\
T
\end{array}\right),\left(\begin{array}{c}
f\\
0\\
0
\end{array}\right)\right)\in\partial_{0,\rho}\left(\begin{array}{ccc}
M & 0 & 0\\
0 & L^{-1} & -L^{-1}B^{*}\\
0 & -BL^{-1}\; & D^{-1}+BL^{-1}B^{*}
\end{array}\right)+\left(\begin{array}{ccc}
0 & 0 & -\Dive\\
0 & g & 0\\
-\Grad & 0 & 0
\end{array}\right),
\]

which fits into our general framework. The operator 
\[
M_{0}(t):=\left(\begin{array}{ccc}
M(t) & 0 & 0\\
0 & L^{-1}(t) & -L(t)^{-1}B^{*}\\
0 & -BL(t)^{-1}\; & D(t)^{-1}+BL(t)^{-1}B^{*}
\end{array}\right)
\]

satisfies the solvability conditions (a) and (b). Concerning (c) and
(d), we observe that by means of a symmetric Gauß step

\[
\left(\begin{array}{cc}
L^{-1}(t) & -L(t)^{-1}B^{*}\\
-BL(t)^{-1}\; & D(t)^{-1}+BL(t)^{-1}B^{*}
\end{array}\right)
\]

is strictly positive definite if and only if
\[
\left(\begin{array}{cc}
L^{-1}(t) & 0\\
0 & D(t)^{-1}
\end{array}\right)
\]

is strictly positive definite, which holds by our assumptions on $L$
and $D$. The maximal monotonicity of

\[
A:=\left(\begin{array}{ccc}
0 & 0 & -\Dive\\
0 & g & 0\\
-\Grad & 0 & 0
\end{array}\right)=\left(\begin{array}{ccc}
0 & 0 & -\Dive\\
0 & 0 & 0\\
-\Grad & 0 & 0
\end{array}\right)+\left(\begin{array}{ccc}
0 & 0 & 0\\
0 & g & 0\\
0 & 0 & 0
\end{array}\right)
\]

can be obtained by imposing suitable boundary conditions on $v$ and
$T$ in order to make 
\[
\left(\begin{array}{ccc}
0 & 0 & -\Dive\\
0 & 0 & 0\\
-\Grad & 0 & 0
\end{array}\right)
\]
skew-selfadjoint and using \prettyref{cor:bounded_pert} (compare
Example 4.1).

\end{document}